\documentclass{article}

\usepackage[utf8]{inputenc}
\usepackage{hyperref}
\usepackage{float}

\usepackage[pdftex]{graphicx}
\usepackage{color}
\usepackage{rotating} 

\usepackage{amsmath}
\usepackage{amsfonts}
\usepackage{amssymb}
\usepackage{amsthm}
\usepackage{mathtools}

\usepackage{framed}
 \usepackage{paralist}
\usepackage[shortlabels]{enumitem}

\usepackage{thmtools}
\usepackage{thm-restate}

\usepackage{hyperref}

\usepackage{cleveref}
\usepackage{tikz}
\usepackage{calculator}

\usepackage{adjustbox}
\usepackage{longtable}

\newtheorem{thm}{Theorem}[section]
\newtheorem{theorem}[thm]{Theorem}
\newtheorem{lemma}[thm]{Lemma}

\newtheorem{remark}[thm]{Remark}
\newtheorem{corollary}[thm]{Corollary}

\theoremstyle{definition}

\newtheorem{definition}[thm]{Definition}
\newtheorem{notation}[thm]{Notation}
\newtheorem{example}[thm]{Example}

\theoremstyle{definition}

\newcommand\Z{\mathbb{Z}}

\newcommand\N{\mathbb{N}}
\newcommand\R{\mathbb{R}}

\newcommand\D{\mathcal{D}}

\renewcommand\O{\mathcal{O}}
\newcommand\eps{\varepsilon}
\newcommand{\tcs}{\setlength{\tabcolsep}{-0.7em}}
\newcommand{\linspan}[1]{\left< #1 \right>}
\newcommand{\ray}[1]{\R_+ #1 }
\newcommand{\Onreg}{O'}

\DeclareMathOperator{\conv}{conv}

\DeclareMathOperator{\C}{C}

\DeclareMathOperator{\Dih}{D}
\DeclareMathOperator{\T}{T}
\DeclareMathOperator{\Oc}{O}
\DeclareMathOperator{\I}{I}
\DeclareMathOperator{\G}{G}

\DeclareMathOperator{\Cf}{Cf} 
\DeclareMathOperator{\RP}{RP} 
\DeclareMathOperator{\HP}{HP} 
\DeclareMathOperator{\THP}{THP} 
\DeclareMathOperator{\TP}{TP} 
\DeclareMathOperator{\BP}{BP} 
\DeclareMathOperator{\CS}{CS} 
\DeclareMathOperator{\CC}{CC} 
\DeclareMathOperator{\ver}{vert}
\DeclareMathOperator{\relint}{relint}
\DeclareMathOperator{\normals}{normals} 
\DeclareMathOperator{\Symm}{Symm} 

\newcommand{\f}{\mathit{f}} 
\newcommand{\F}{\mathit{F}} 
\newcommand{\lt}{left type}
\newcommand{\rt}{right type}
\newcommand{\bt}{base} 

\newcommand{\defRoot}[1]{
\def\root{#1}
}
\newcommand{\defLR}[2]{
\def\cL{#1}
\def\cR{#2}
}
\newcommand{\defLRTX}[4]{
\def\cL{#1}
\def\cR{#2}
\def\cT{#3}
\def\cX{#4}
}
\newcommand{\defT}[1]{
\def\cT{#1}
}
\newcommand{\certLR}{
\begin{adjustbox}{valign=t}
\begin{tikzpicture}
    \node at (0,0.5) {\root};
		\node (L) at (0,0) {\cL};
		\node at (0,-0.3) {\tiny\bf L};
		\node (R) at (1,-1) {\cR} edge[bend right = 20] (L);
		\node at (1,-1.3) {\tiny\bf R};
\end{tikzpicture}
\end{adjustbox}
}
\newcommand{\certRL}{
\begin{adjustbox}{valign=t}
\begin{tikzpicture}
    \node at (0,0.5) {\root};
		\node (R) at (0,0) {\cR};
		\node at (0,-0.3) {\tiny\bf R};
		\node (L) at (-1,-1) {\cL} edge[bend left = 20] (R);
		\node at (-1,-1.3) {\tiny\bf L};
\end{tikzpicture}
\end{adjustbox}
}
\newcommand{\certTri}{
\begin{adjustbox}{valign=t}
\begin{tikzpicture}
  \node at (0,0.5) {\root};
	\node (T) at (0,0) {\cT};
	\node at (0,-0.3) {\tiny\bf X};
	\node (L) at (-1,-1) {\cL} edge[bend left = 20] (T);
	\node at (-1,-1.3) {\tiny\bf L};
	\node (R) at (1,-1) {\cR} edge[bend right = 20] (T);
	\node at (1,-1.3) {\tiny\bf R};
	\node (X) at (0,-2) {\cX} edge[bend left = 20] (R);
	\node at (0,-2.3) {\tiny\bf X};
	\draw (L) edge[bend left = 20] (X);
\end{tikzpicture}
\end{adjustbox}
}

\newcommand{\certBtikz}{
\begin{tikzpicture}
  \node at (0,0.5) {\root};
	\node (T) at (0,0) {\cT};
	\node at (0,-0.3) {\tiny\bf B};
\end{tikzpicture}
}
\newcommand{\certB}{
\begin{adjustbox}{valign=t}
\certBtikz
\end{adjustbox}
}

\newcommand{\RPpic}{
\begin{center}
\begin{tikzpicture}
\coordinate (A) at (-1,-1/3*\sqrtthree);
\coordinate (B) at (1,-1/3*\sqrtthree);
\coordinate (C) at (0,2/3*\sqrtthree);
\coordinate (Ap) at (-0.5,-0.5*1/3*\sqrtthree);
\coordinate (Bp) at (0.5, -0.5*1/3*\sqrtthree);
\coordinate (Cp) at (0,0.5*2/3*\sqrtthree);
\draw (A)--(B)--(C)--(A);
\draw (Ap)--(Bp)--(Cp)--(Ap);
\draw (A)--(Ap);
\draw (B)--(Bp);
\draw (C)--(Cp);
\end{tikzpicture}
\hspace{2cm}
\begin{tikzpicture}
\coordinate (A) at (-1,-1);
\coordinate (B) at (1,-1);
\coordinate (C) at (1,1);
\coordinate (D) at (-1,1);
\coordinate (Ap) at (-0.5,-0.5);
\coordinate (Bp) at (0.5, -0.5);
\coordinate (Cp) at (0.5,0.5);
\coordinate (Dp) at (-0.5,0.5);
\draw (A)--(B)--(C)--(D)--(A);
\draw (Ap)--(Bp)--(Cp)--(Dp)--(Ap);
\draw (A)--(Ap);
\draw (B)--(Bp);
\draw (C)--(Cp);
\draw (D)--(Dp);
\end{tikzpicture}
\end{center}}

\newcommand{\TPpic}{
\begin{center}
\begin{tikzpicture}
\coordinate (A) at (-1.5*1,-1.5*1/3*\sqrtthree);
\coordinate (B) at (1.5*1,-1.5*1/3*\sqrtthree);
\coordinate (C) at (0,1.5*2/3*\sqrtthree);
\begin{scope}[rotate = 60]
\coordinate (Ap) at (-0.5,-0.5*1/3*\sqrtthree);
\coordinate (Bp) at (0.5, -0.5*1/3*\sqrtthree);
\coordinate (Cp) at (0,0.5*2/3*\sqrtthree);
\end{scope}
\draw (A)--(B)--(C)--(A);
\draw (Ap)--(Bp)--(Cp)--(Ap);
\draw (A)--(Cp);
\draw (A)--(Ap);
\draw (B)--(Ap);
\draw (B)--(Bp);
\draw (C)--(Bp);
\draw (C)--(Cp);
\end{tikzpicture}
\hspace{2cm}
\begin{tikzpicture}
\coordinate (A) at (-1,-1);
\coordinate (B) at (1,-1);
\coordinate (C) at (1,1);
\coordinate (D) at (-1,1);
\begin{scope}[rotate = 45]
\coordinate (Ap) at (-0.5,-0.5);
\coordinate (Bp) at (0.5, -0.5);
\coordinate (Cp) at (0.5,0.5);
\coordinate (Dp) at (-0.5,0.5);
\end{scope}
\draw (A)--(B)--(C)--(D)--(A);
\draw (Ap)--(Bp)--(Cp)--(Dp)--(Ap);
\draw (A)--(Dp);
\draw (A)--(Ap);
\draw (B)--(Ap);
\draw (B)--(Bp);
\draw (C)--(Bp);
\draw (C)--(Cp);
\draw (D)--(Cp);
\draw (D)--(Dp);
\end{tikzpicture}
\end{center}}

\newcommand{\BPpic}{
\begin{center}
\begin{tikzpicture}
\DIVIDE{\sqrtthree}{3}{\solution}
\coordinate (A) at (-1,-1/3*\sqrtthree);
\coordinate (B) at (1,-1/3*\sqrtthree);
\coordinate (C) at (0,2/3*\sqrtthree);
\coordinate (Ap1) at (-0.5*2/3,0);
\coordinate (Ap2) at (-0.5*1/3,-0.5*1/3*\sqrtthree);
\coordinate (Bp1) at (0.5*1/3, -0.5*1/3*\sqrtthree);
\coordinate (Bp2) at (0.5*2/3, 0);
\coordinate (Cp1) at (0.5*1/3,0.5*1/3*\sqrtthree);
\coordinate (Cp2) at (-0.5*1/3,0.5*1/3*\sqrtthree);
\draw (A)--(B)--(C)--(A);
\draw (Ap1)--(Ap2)--(Bp1)--(Bp2)--(Cp1)--(Cp2)--(Ap1);
\draw (A)--(Ap1);
\draw (A)--(Ap2);
\draw (B)--(Bp1);
\draw (B)--(Bp2);
\draw (C)--(Cp1);
\draw (C)--(Cp2);
\end{tikzpicture}
\hspace{2cm}
\begin{tikzpicture}
\coordinate (A) at (-1,-1);
\coordinate (B) at (1,-1);
\coordinate (C) at (1,1);
\coordinate (D) at (-1,1);
\coordinate (Ap1) at (-0.5,-0.5*1/3);
\coordinate (Ap2) at (-0.5*1/3,-0.5);
\coordinate (Bp1) at (0.5*1/3, -0.5);
\coordinate (Bp2) at (0.5, -0.5*1/3);
\coordinate (Cp1) at (0.5,0.5*1/3);
\coordinate (Cp2) at (0.5*1/3,0.5);
\coordinate (Dp1) at (-0.5*1/3,0.5);
\coordinate (Dp2) at (-0.5,0.5*1/3);
\draw (A)--(B)--(C)--(D)--(A);
\draw (Ap1)--(Ap2)--(Bp1)--(Bp2)--(Cp1)--(Cp2)--(Dp1)--(Dp2)--(Ap1);
\draw (A)--(Ap1);
\draw (A)--(Ap2);
\draw (B)--(Bp1);
\draw (B)--(Bp2);
\draw (C)--(Cp1);
\draw (C)--(Cp2);
\draw (D)--(Dp1);
\draw (D)--(Dp2);
\end{tikzpicture}
\end{center}}

\newcommand{\HPpic}{
\begin{center}
\begin{tikzpicture}
\DIVIDE{\sqrtthree}{3}{\solution}
\coordinate (A) at (-0.5*1,-0.5*1/3*\sqrtthree);
\coordinate (B) at (0.5*1,-0.5*1/3*\sqrtthree);
\coordinate (C) at (0,0.5*2/3*\sqrtthree);
\coordinate (Ap1) at (-1.5*2/3,0);
\coordinate (Ap2) at (-1.5*1/3,-1.5*1/3*\sqrtthree);
\coordinate (Bp1) at (1.5*1/3, -1.5*1/3*\sqrtthree);
\coordinate (Bp2) at (1.5*2/3, 0);
\coordinate (Cp1) at (1.5*1/3,1.5*1/3*\sqrtthree);
\coordinate (Cp2) at (-1.5*1/3,1.5*1/3*\sqrtthree);
\draw (A)--(B)--(C)--(A);
\draw (Ap1)--(Ap2)--(Bp1)--(Bp2)--(Cp1)--(Cp2)--(Ap1);
\draw (A)--(Ap1);
\draw (A)--(Ap2);
\draw (B)--(Bp1);
\draw (B)--(Bp2);
\draw (C)--(Cp1);
\draw (C)--(Cp2);
\end{tikzpicture}
\hspace{2cm}
\begin{tikzpicture}
\coordinate (A) at (-0.5*1,-0.5*1);
\coordinate (B) at (0.5*1,-0.5*1);
\coordinate (C) at (0.5*1,0.5*1);
\coordinate (D) at (-0.5*1,0.5*1);
\coordinate (Ap1) at (-1,-1/3);
\coordinate (Ap2) at (-1/3,-1);
\coordinate (Bp1) at (1/3, -1);
\coordinate (Bp2) at (1, -1/3);
\coordinate (Cp1) at (1,1/3);
\coordinate (Cp2) at (1/3,1);
\coordinate (Dp1) at (-1/3,1);
\coordinate (Dp2) at (-1,1/3);
\draw (A)--(B)--(C)--(D)--(A);
\draw (Ap1)--(Ap2)--(Bp1)--(Bp2)--(Cp1)--(Cp2)--(Dp1)--(Dp2)--(Ap1);
\draw (A)--(Ap1);
\draw (A)--(Ap2);
\draw (B)--(Bp1);
\draw (B)--(Bp2);
\draw (C)--(Cp1);
\draw (C)--(Cp2);
\draw (D)--(Dp1);
\draw (D)--(Dp2);
\end{tikzpicture}
\end{center}}

\newcommand{\THPpic}{
\begin{center}
\begin{tikzpicture}
\begin{scope}[rotate = 30]
\coordinate (A) at (-0.5*1,-0.5*1/3*\sqrtthree);
\coordinate (B) at (0.5*1,-0.5*1/3*\sqrtthree);
\coordinate (C) at (0,0.5*2/3*\sqrtthree);
\end{scope}
\coordinate (Ap1) at (-1.5*2/3,0);
\coordinate (Ap2) at (-1.5*1/3,-1.5*1/3*\sqrtthree);
\coordinate (Bp1) at (1.5*1/3, -1.5*1/3*\sqrtthree);
\coordinate (Bp2) at (1.5*2/3, 0);
\coordinate (Cp1) at (1.5*1/3,1.5*1/3*\sqrtthree);
\coordinate (Cp2) at (-1.5*1/3,1.5*1/3*\sqrtthree);
\draw (A)--(B)--(C)--(A);
\draw (Ap1)--(Ap2)--(Bp1)--(Bp2)--(Cp1)--(Cp2)--(Ap1);

\draw (A)--(Ap1);
\draw (A)--(Ap2);
\draw (A)--(Bp1);
\draw (B)--(Bp1);
\draw (B)--(Bp2);
\draw (B)--(Cp1);
\draw (C)--(Cp1);
\draw (C)--(Cp2);
\draw (C)--(Ap1);
\end{tikzpicture}
\hspace{2cm}
\begin{tikzpicture}
\begin{scope}[rotate = 22.5]
\coordinate (A) at (-0.5*1,-0.5*1);
\coordinate (B) at (0.5*1,-0.5*1);
\coordinate (C) at (0.5*1,0.5*1);
\coordinate (D) at (-0.5*1,0.5*1);
\end{scope}
\coordinate (Ap1) at (-1,-1/3);
\coordinate (Ap2) at (-1/3,-1);
\coordinate (Bp1) at (1/3, -1);
\coordinate (Bp2) at (1, -1/3);
\coordinate (Cp1) at (1,1/3);
\coordinate (Cp2) at (1/3,1);
\coordinate (Dp1) at (-1/3,1);
\coordinate (Dp2) at (-1,1/3);
\draw (A)--(B)--(C)--(D)--(A);
\draw (Ap1)--(Ap2)--(Bp1)--(Bp2)--(Cp1)--(Cp2)--(Dp1)--(Dp2)--(Ap1);
\draw (A)--(Ap1);
\draw (A)--(Ap2);
\draw (A)--(Bp1);
\draw (B)--(Bp1);
\draw (B)--(Bp2);
\draw (B)--(Cp1);
\draw (C)--(Cp1);
\draw (C)--(Cp2);
\draw (C)--(Dp1);
\draw (D)--(Dp1);
\draw (D)--(Dp2);
\draw (D)--(Ap1);
\end{tikzpicture}
\end{center}}

\newcommand{\HasseDiagramOfOrthogonalGroups}{
\begin{center}
\begin{tikzpicture}[scale = 1.5]
\node (C2) at (1,0) {$\C_2$};
\node (C3) at (2,1) {$\C_3$};
\node (C5) at (3,1) {$\C_5$};
\node (C4) at (0,1) {$\C_4$};
\node (D2) at (1,1) {$\Dih_2$};
\node (D5) at (3,2) {$\Dih_5$};
\node (D4) at (0,2) {$\Dih_4$};
\node (T) at (2,2) {$\T$};
\node (O) at (1,3) {$\Oc$};
\node (I) at (2.5,3) {$\I$};

\draw (C2)--(C4)--(D4) -- (O);
\draw (C2) -- (D2) -- (T) -- (O);
\draw (D2) -- (D4);
\draw (C3)--(T)--(O);
\draw (T)--(I);
\draw (C5)--(D5)--(I);
\end{tikzpicture}
\end{center}}

\SQUAREROOT{3}{\sqrtthree}

\begin{document}

\author{Maren H. Ring\footnote{University of Rostock, maren.ring@uni-rostock.de}, Robert Schüler\footnote{University of Rostock, robert.schueler1989@gmail.com}}


%
%

\title{f-vectors of 3-polytopes symmetric under rotations and rotary reflections}
\date{\today}

\maketitle
\defRoot{$\f$}


















\begin{abstract}
The $f$-vector of a polytope consists of the numbers of its $i$-dimensional faces. An open field of study is the characterization of all possible $f$-vectors. It has been solved in three dimensions by Steinitz in the early 19th century. We state a related question, i.e. to characterize $f$-vectors of three dimensional polytopes respecting a symmetry, given by a finite group of matrices. 
We give a full answer for all three dimensional polytopes that are symmetric with respect to a finite rotation or rotary reflection group. 
We solve these cases constructively by developing tools that generalize Steinitz's approach. 
\end{abstract}

\section{Introduction}

%
%
%
%
%
%
%
%
%
%


There are many studies of $f$-vectors in higher dimensions, for example see \cite{grunbaum2003convex}, \cite{mcmullen1971minimum_vertices}, \cite{barnette1972inequalities}, \cite{barnette1973proof_of_lower_bound_conjecture}, \cite{bayer1984Counting_faces_and_chains}, \cite{stanley80simplicial}, \cite{stanley1985simplicial} or \cite{bayer1985Generalized_Dehn_Sommerville}. The set of $f$-vectors of four dimensional polytopes has been studied in \cite{barnette1973projections_four_polytopes}, \cite{barnette1974projection}, \cite{altshuler1985complete_enumeration}, \cite{bayer1987extended_f_vectors}, \cite{eppstein2003fat_and_fatter}, \cite{ziegler2002face_numbers} and \cite{brinkmann2018small_f_vectors}. Some insights about $f$-vectors of centrally symmetric polytopes are given in \cite{Barany1982Borsuks_theorem}, \cite{campo-neuen2006on_toric_h_vectors}, \cite{stanley1987centrally_symmetric_simplicial} and \cite{freij2013face_numbers_of_csp}.
It is still an open question, even in three dimensions, what the $f$-vectors of \emph{symmetric polytopes} are. This question will be partially answered in this paper.

In particular, given a finite $3\times 3$ matrix group $G$, we ask to determine the set
$\F(G)$ of vectors $(\f_0,\f_2)\in \N\times\N$ such that there is a polytope $P$ symmetric under $G$ (i.e. $A\cdot P = P$ for all $A\in G$) with $\f_0$ vertices and $\f_2$ facets (we omit the number of edges by the Euler-equation). In this paper we give an answer for all groups that don't contain a reflection summarized in the following theorem. For a detailed explanation of the mentioned groups see Theorem \ref{thm:finite_orthogonal_groups}.

\begin{theorem}[main theorem]\label{thm:main}
Let $\F$ be the set of $\f$-vectors of three dimensional polytopes (ommiting the number of edges) and for $M\subset\N\times\N$ we use $M^\diamond \coloneqq \{(y,x) \ : \ (x,y)\in M\}$ as well as $\equiv$ to denote component wise congruence.
The $f$-vectors of symmetric polytopes under rotation groups can be classified as follows:
\begin{align*}
\F(\C_n) &= \{\f\in\F \ : \ \f\equiv (1,1) \mod n\} ^\diamond\\
&\cup \{\f = (\f_0,\f_2) \in \F \ : \ \f\equiv (0,2), 2\f_0 - \f_2 \geq 2n-2 \mod n\}^\diamond \textnormal{ for } n > 2,\\
\F(\C_2) &= \F,\\
\F(\Dih_d) &= \{\f\in\F \ : \ \f\equiv (0,2),(2,d) \mod 2d\}^\diamond \\
&\cup \{\f = (\f_0,\f_2)\in\F \ : \f\equiv (0,d+2), (d,d+2) \mod 2d , 2\f_0 - \f_2 \geq 3d-2\}^\diamond\\
&\textnormal{ for } d>2\\
\F(\Dih_2) &= \{\f\in\F \ : \ \f\equiv (0,0),(0,2),(2,2) \mod 4\}^\diamond \setminus\{(6,6)\},\\
 \F(\T) &= \{\f\in\F \ : \ \f\equiv (0,2),(0,8),(4,4),(4,10),(6,8) \mod 12\}^\diamond,\\
\F(\Oc) &= \{\f\in\F \ : \ \f\equiv (0,2),(0,14),(6,8),(6,20),(8,18),(12,14) \mod 24\}^\diamond,\\
\F(\I) &= \{\f\in\F \ : \ \f\equiv (0,2),(0,32),(12,20), (12,50), \\
 & \hspace{19em} (20,42),(30,32) \mod 60\}^\diamond .
\end{align*}
For rotary reflection groups, the $f$-vectors can be classified as:
\begin{align*}
\F(\G_d) &= \{\f\in\F \ : \ \f\equiv (0,2) \mod 2d\}^\diamond \textnormal{ for } d > 2,\hfill \\
\F(\G_2) &= \{\f\in\F \ : \ \f\equiv (0,0),(0,2) \mod 4\}^\diamond.\\
\F(\G_1) &= \{\f\in\F \ : \ \f\equiv (0,0) \mod 2\}^\diamond\setminus \{(4,4),(6,6)\}
\end{align*}


In this paper we generalize the elementary approach of Steinitz (see \cite{steinitz1906polyederrelationen}). We start by introducing some fundamental terms and concepts relevant to this work in Section \ref{sec:preliminaries}.
The coarser structure of $F(G)$ is due to the composition of orbits that the group $\G$ admits. This can be described in general and will be shown Section \ref{sec:conditions}, especially in Lemma~\ref{lem:f_vectors_mod_n}. The extra restrictions arise from certain structures of facets and vertices, e.g. a facet on a 6-fold rotation axis must have at least 6 vertices, which forces the polytope to be 'further away' from being simplicial.

The main difficulty in characterizing $F(G)$ is the construction of $G$-symmetric polytopes with a given $f$-vector. 
In  Section \ref{sec:base} we introduce so called \emph{base polytopes}, symmetric polytopes that can be used to generate an infinite class of $f$-vectors.
Since the operations on base polytopes produce $f$-vectors in the same congruence class, we divide the set of possible $f$-vectors in $F(G)$ into several coarser integer cones. 
To certify the existence of all $f$-vectors in one of these coarser integer cones
 we introduce four types of certificates in Section \ref{sec:certificates}. In Corollary \ref{cor:certificates_work} we describe for which  $f$-vectors certificates are needed to obtain all $f$-vectors conjectured to be in $F(G)$. To give these certificates we need to find symmetric polytopes with 'small' $f$-vector. To this end, in Section \ref{sec:constructions}, we  introduce  useful constructions on polytopes that change the $f$-vectors, but preserve the symmetry. As starting points, we then give a list of some well known polytopes taken from the Platonic, the Archimedean and their duals, the Catalan solids. 
In Section \ref{sec:characterization} we are finally able to connect the theory with explicit constructions of polytopes to prove Theorem \ref{thm:main}.
 Lastly, we conclude the paper with some open questions and conjectures in Section \ref{sec:open}.
\end{theorem}

\section{Preliminaries}\label{sec:preliminaries}

We start by introducing some fundamental terms and concepts relevant to this work.
A polytope is the convex hull of finitely many points in $\R^n$. A face of a polytope $P$ is the intersection of $P$ with a hyperplane that contains no  points of the relative interior of $P$. The polytope $P$ itself and the empty set are often considered as non-proper faces of $P$ as well, but are irrelevant for the study of $f$-vectors, since there is always exactly one of each. The dimension of a face is the dimension of its affine hull.

Let $P\subset\R^3$ be a 3-dimensional polytope. For $i\in \{0, 1,2\}$, we denote by $f_i(P)$ the number of $i$-dimensional faces of $P$. We call a facet \emph{simplicial} if it is triangular and a vertex \emph{simple}  if it has degree three. A polytope is called simplicial if all of its facets are simplicial and it is called simple  if all of its vertices are simple. Furthermore, a $(i,j)$-flag of $P$ for $i\leq j$ consists of an $i$-dimensional face $F_i$ and a $j$-dimensional face $F_j$ of $P$ such that $F_i\subseteq F_j$. The number of all $(i,j)$-flags of $P$ is denoted by $f_{i,j}(P)$.
The following result 
follows directly from the Dehn-Sommerville Equations (see \cite[Section 9.2]{grunbaum2003convex}).
It establishes certain dependencies among the numbers of $i$-dimensional faces of a 3-dimensional polytope:

\begin{theorem}\label{thm:Euler_Steinitz}
For any 3-dimensional polytope $P$ we have:
\begin{enumerate}[label=(\arabic*)]
\item $f_0(P) - f_1(P) + f_2(P) = 2$, \label{fml:Euler}
\item $2f_0(P) - f_2(P) \geq 4$,\label{it:simplicial}
\item $2f_2(P) - f_0(P) \geq 4$.\label{it:simple}
\end{enumerate}
We have equality in Equation \ref{it:simplicial} (Equation \ref{it:simple}) if and only if $P$ is simplicial (simple).
\end{theorem}

It thus suffices to know two of the three numbers $f_0(P)$, $f_1(P)$ and $f_2(P)$;  the missing one can be computed using Theorem \ref{thm:Euler_Steinitz}, Equation \ref{fml:Euler}. That leads us to the following definition of an $f$-vector that differs slightly from the definition for higher dimensional polytopes, where all numbers $f_i(P)$ occur in the $f$-vector.
\begin{definition}
Let $P$ be a 3-dimensional polytope. We define the \emph{$f$-vector}  of $P$ to be $f(P) = (f_0(P), f_2(P))$. 
\end{definition} 

Recall, that we denote by $F$ the set of all possible $f$-vectors of 3-dimensional polytopes, i.e.
\begin{equation*}
F=\{f\in \Z^2 \ : \ \exists  \textnormal{ polytope } P \textnormal{ with } \dim(P)=3 \textnormal{ and } f(P)=f \}.
\end{equation*} 

It turns out that the conditions of Theorem \ref{thm:Euler_Steinitz} are sufficient for a characterization of $F$. 

\begin{theorem}[{\cite[Section 10.3]{grunbaum2003convex}}] 
We have 
\begin{align*}
\F = \{(f_0,f_2)\in \Z^2 \ : \ 2f_0 - f_2 \geq 4  \textnormal{ and } 2f_2 - f_0 \geq 4 \},
\end{align*}
i.e. the set of all possible $f$-vectors of 3-dimensional polytopes is a translated integer cone as shown in Figure \ref{fig:f_all}.
\end{theorem}
\begin{figure}
\begin{center}
\scalebox{0.3}{\includegraphics{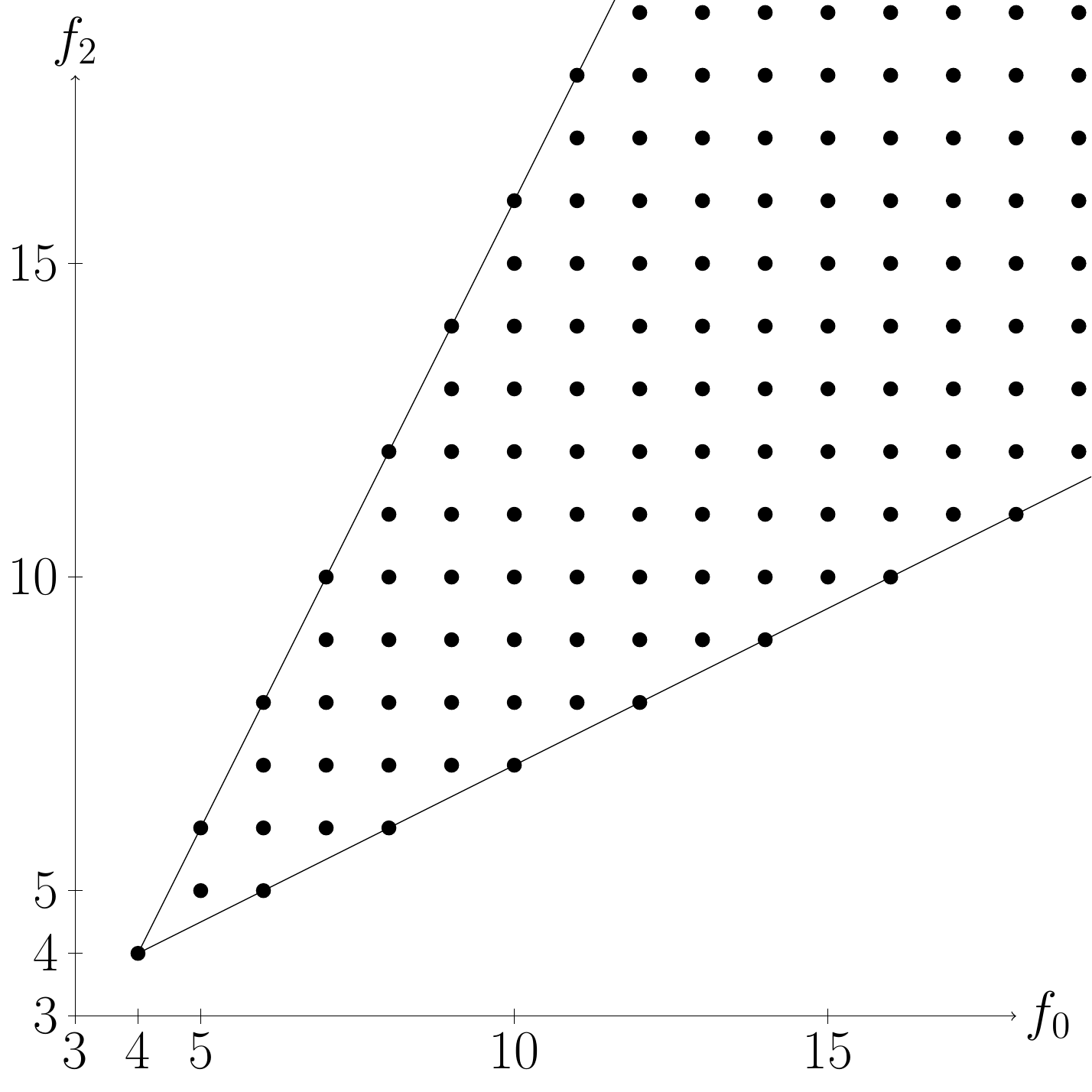}}
\caption{The set $F$ of all $f$-vectors of 3-dimensional polytopes.}\label{fig:f_all}
\end{center}
\end{figure}


Let $G$ be a matrix group. We say that a polytope $P$ is \emph{$G$-symmetric} if $G$ acts on $P$, i.e. $G\cdot P = P$.  The set of all possible $f$-vectors of $G$-symmetric polytopes is denoted by
$$\F(G) = \{f\in\Z^2 \ : \ \exists \  G\text{-symmetric polytope } P \text{ s.t. } \ f(P) = f\}.$$

By the following well known fact, it is sufficient to study orthogonal groups.
\begin{lemma}
If $G$ is a finite subgroup of $GL_3(\R)$, then there is an inner product $\linspan{\cdot,\cdot}_G$ such that $G$ is an orthogonal group with respect to $\linspan{\cdot,\cdot}_G$.
\end{lemma}

Hence, for any finite matrix group $G$ there is an orthogonal group $G'$ with $\F(G) = \F(G')$ and we can reduce our study to finite orthogonal groups.  Luckily, the finite orthogonal groups are well known. In fact, there are only finitely many families. A full list can be found in \cite{benson1985finite_reflection_groups}. Here, we will state the characterization of all finite rotation and rotary reflection groups, since these are our main subject of interest. The remaining finite orthogonal groups are the ones that contain reflections and are discussed in Section \ref{sec:open}.

\begin{theorem}[{\cite[Theorem 2.5.2]{benson1985finite_reflection_groups}}]\label{thm:finite_orthogonal_groups}
If $G$ is a finite orthogonal subgroup of $GL_3(\R)$ consisting only of rotations or rotary reflections, then it is isomorphic to one of the following:
\begin{enumerate}
\item the axis-rotation group $C_n$ 
\item the dihedral rotation group $\Dih_d$
\item the rotation group of the regular tetrahedron $\T$
\item the rotation group of the regular octahedron $\O$
\item the rotation group of the regular icosahedron $\I$
\item the rotary reflection group $\G_d$ of order $2d$ (generated by a product of a rotation and a reflection) \label{it:rotary}
\end{enumerate}
\end{theorem}

\begin{remark}
Details of the groups will be explained in Section \ref{sec:characterization}.
In the notation of Grove and Benson \ref{thm:finite_orthogonal_groups}.\ref{it:rotary}  corresponds to $C_3^{2d}]C_3^{d}$ for even $d$ and to $(C_3^d)^\ast$ for odd $d$ and is generated by a 'rotary reflection', i.e. the conjunction of a rotation and a reflection on a hyperplane perpendicular to the rotation axis.
\end{remark}

Figure \ref{fig:hasse_diagram} shows a Hasse-diagram of some subgroup relations of these groups.
\begin{figure}
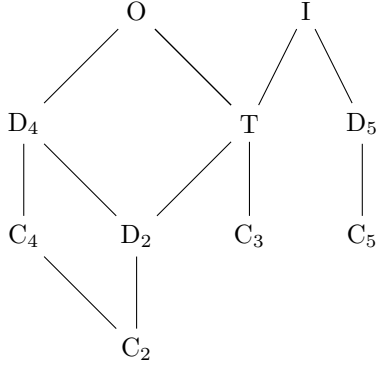

\HasseDiagramOfOrthogonalGroups
\caption{A Hasse diagram of the icosahedral and octahedral rotation group (up to conjugation by orthogonal matrices)}\label{fig:hasse_diagram}
\end{figure}
A detailed description of each group is given in Section \ref{sec:certificates}.

We proceed to discuss certain regularities of the sets $F(G)$.
For any set $M\subset\Z^2$ 
we denote the symmetric set obtained from $M$ as 
\begin{equation*}
M^\diamond = M \cup \{(y,x) \ : \ (x,y)\in M\}.
\end{equation*} 
Considering the $f$-vector of the dual polytope $f(P^\vee) = (f_2(P), f_0(P))$ shows that $\F$ is invariant under the $\diamond$ operation.
The set $F(G)$ is also invariant under the $\diamond$ operation:
Let $P$ be a $G$-symmetric polytope and let $b$ be the barycenter of~$P$. The polytope $P^* = \{x\in\R^3 \ : \ \linspan{x,y} \leq 1 \text{ for all } y\in P-b\}$ is a $G$-symmetric polytope with
$f(P^*) = (f_2(P),f_0(P))$.
We denote by $P^\vee$ any dual polytope of $P$ that is additionally symmetric under the same symmetry group as $P$ (for example $P^*$). Furthermore, note that for $G_1 \geq G_2$ we have $F(G_1)\subseteq F(G_2)$.

Now we have everything to completely understand our main theorem (Theorem \ref{thm:main}).



If we visualize the set $F$ in a two dimensional coordinate system, it looks like a cone translated by $(4,4)$ (cf. Figure~\ref{fig:f_all}).

To show that all integer points in the cone exist as $f$-vectors of 3-polytopes, Steinitz starts with pyramids over $n$-gons, whose $f$-vectors are $(n+1,n+1)$ for $n\in \Z_{>2}$.
 He then shows that by stacking a point on a simplicial facet and by cutting a simple vertex, the $f$-vector of the polytope changes by $+(1,2)$ and $+(2,1)$, respectively. 
The resulting polytope again has simple vertices and simplicial facets, which means that the construction can be repeated. 
With this method, it can be shown that for all  points $f$ in the set $\{(f_0,f_2)\in \Z^2 \colon 2f_0 - f_2 \geq 4 \textnormal{ and } 2f_2 - f_0 \geq 4 \}$ we can construct a  3-polytope $P$ with $f(P)=f$. 
Here, the  pyramids over $n$-gons in a way act as generators. 
In fact, the pyramids over a $3$-, a $4$- and a $5$-gon would already suffice to generate all $f$-vectors by stacking on simplicial facets  and cutting simple vertices. 

Clearly, the sets $F(G)$ for finite symmetry groups $G$ are subsets of $F$. 
Theorem \ref{thm:main} states that  for most groups, $F(G)$ is coarser, since only some values modulo $n$ are allowed. 
For some groups of $\C_n$ and $\Dih_d$ there are also  inequalities restricting some residue classes. The third alteration we observe is that one or several 'small' points are left out as it happens for $\Dih_2$ and $\G_1$. These alterations are illustrated in Figure~\ref{fig:f_divers}.
\begin{figure}
\begin{center}
\scalebox{0.10}{\includegraphics{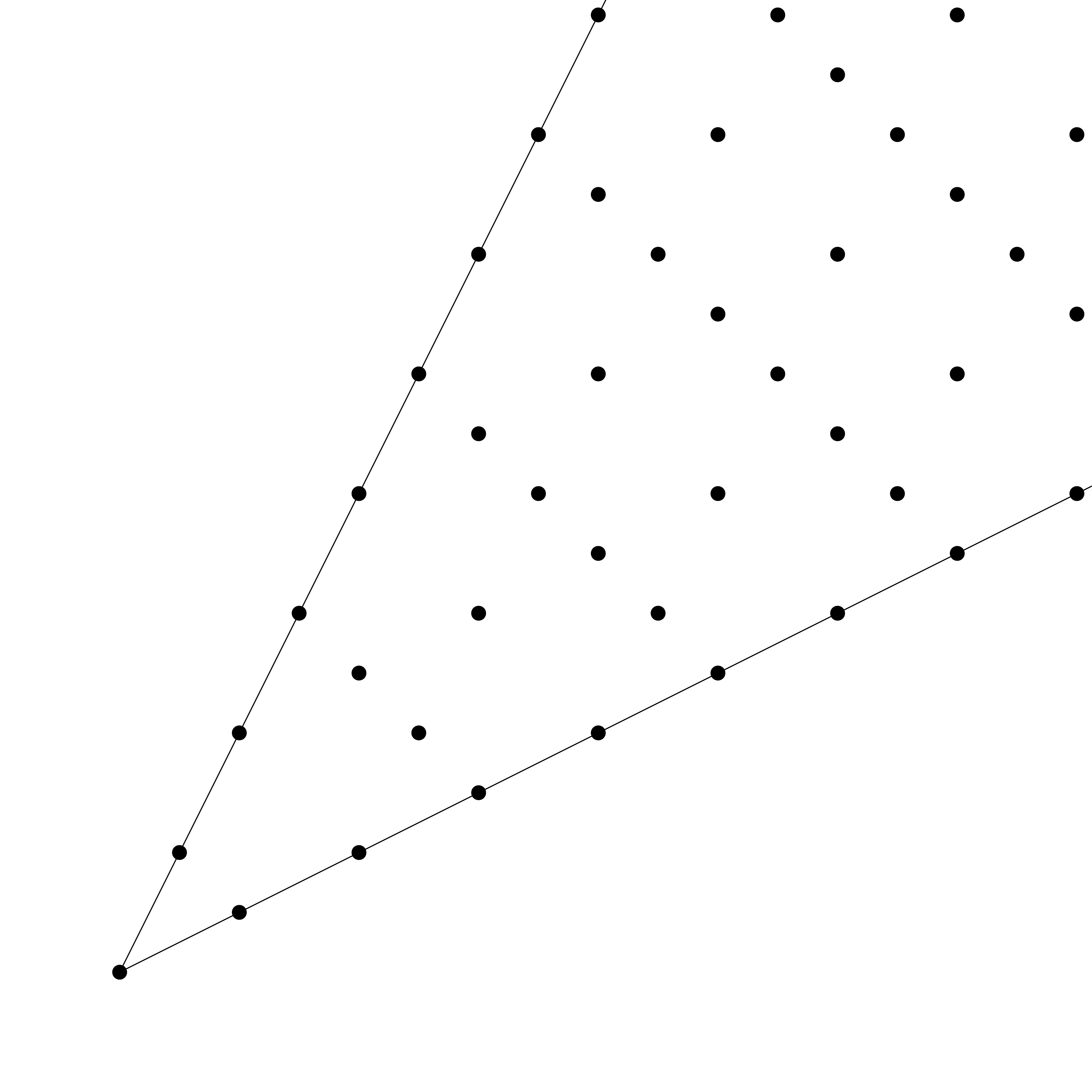}}
\scalebox{0.10}{\includegraphics{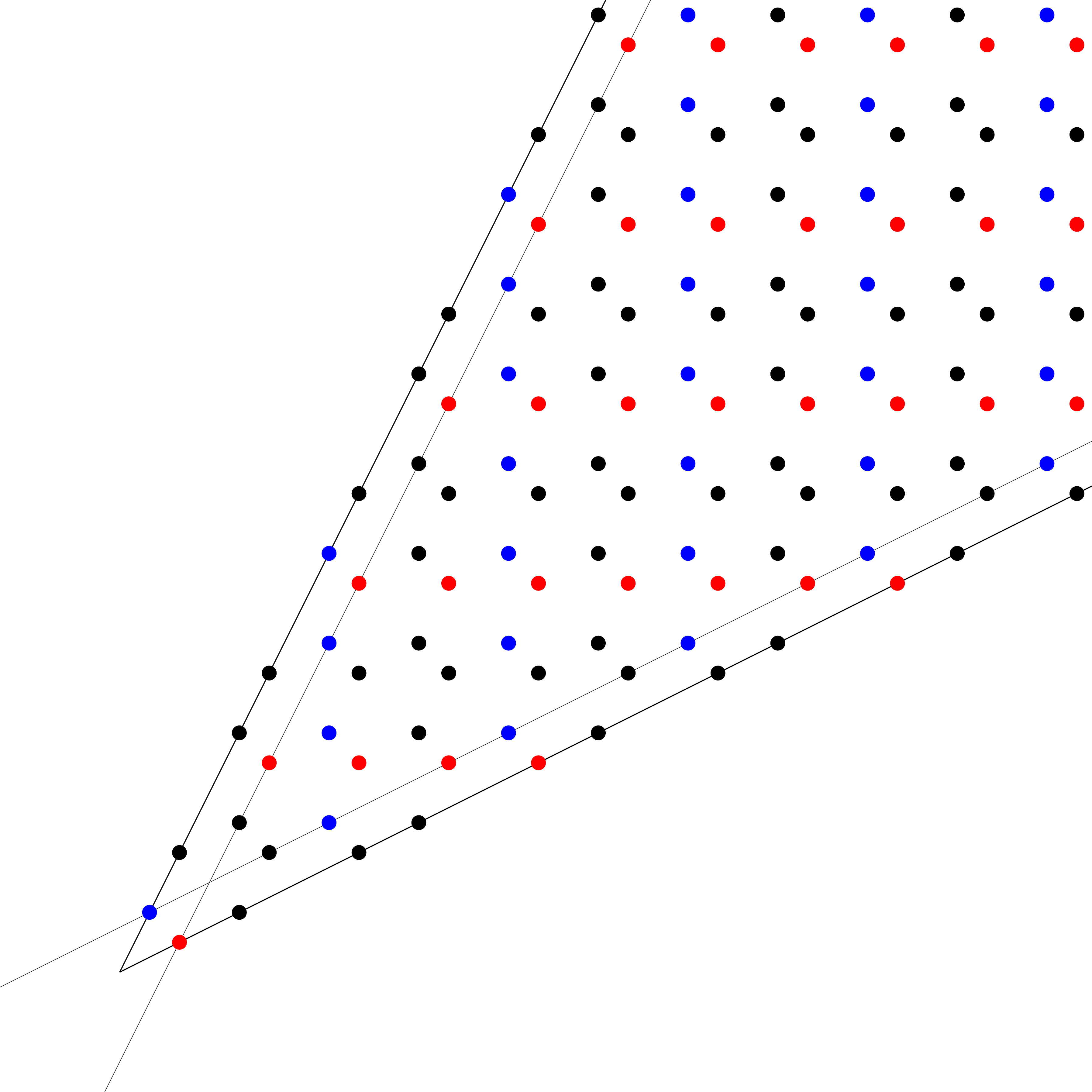}}
\scalebox{0.10}{\includegraphics{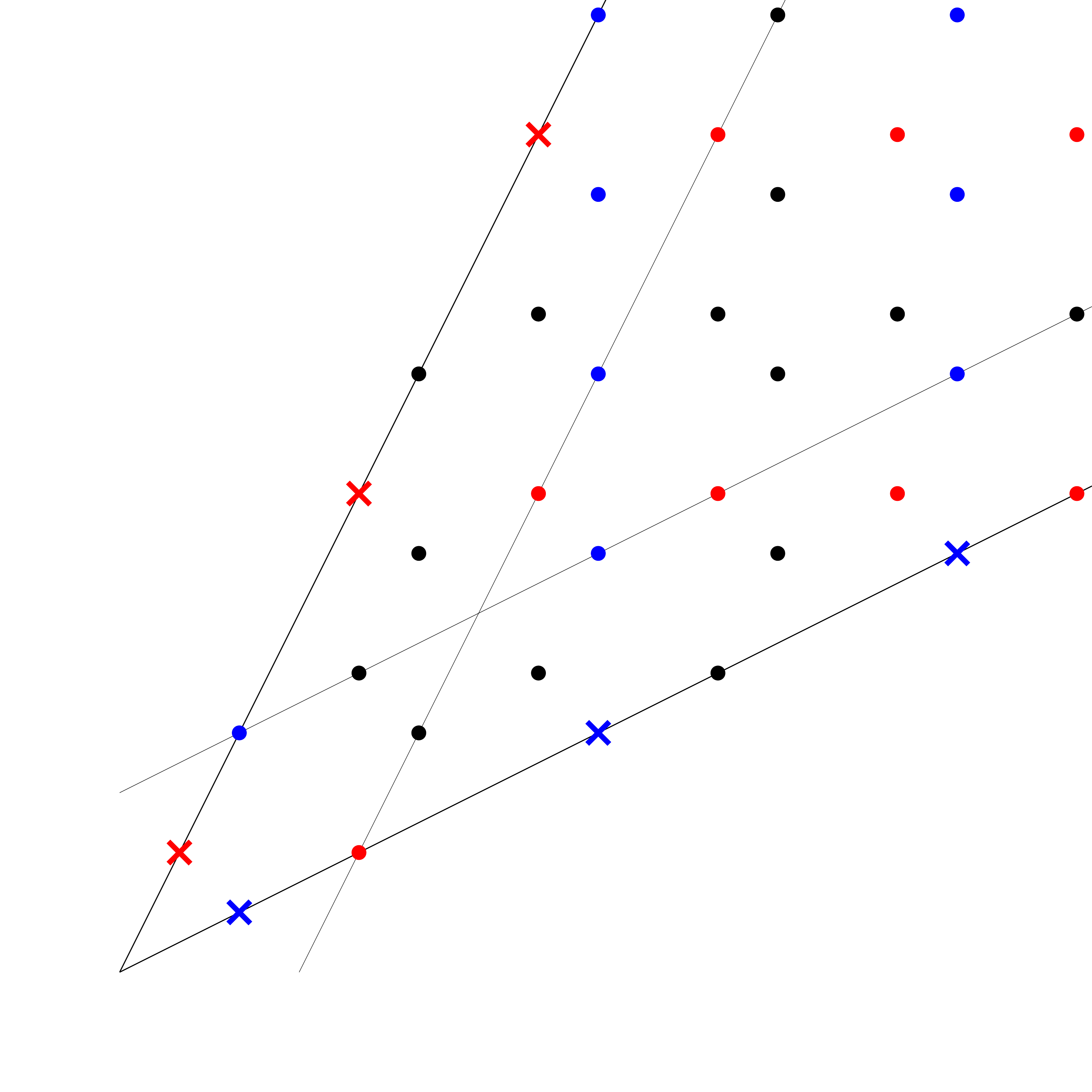}}
\caption{The sets $F(\T)$, $F(\Dih_3)$ and $F(\Dih_6)$ respectively. Different colours indicate different residue classes, crosses indicate points that are not contained in the set due to additional inequalities.}\label{fig:f_divers}
\end{center}
\end{figure}

%

\section{Conditions on $\F(G)$}\label{sec:conditions}

Let $G$ be a finite orthogonal subgroup of $GL_3(\R)$.
In this section we deduce conditions on the sets $\F(G)$ in dependence of the group $G$. These conditions mostly depend on the structures of orbits under the action of $G$ on $\R^3$.

We start with some notation.
A \emph{ray} in $\R^3$ is  a set of the form 
$$\ray{x}=\{\lambda \cdot x \ : \ \lambda>0\} $$
for some  $x\in \R^3\backslash\{0\}$. We  say that $\ray{x}$ is the \emph{ray generated by $x$}.
If we consider two points in the same ray, say $v \in V\backslash \{0\}$ and $\lambda v$ with $\lambda > 0$,  we notice that the orbit polytopes $\conv\{G\cdot v\}$ and $\conv\{G\cdot \lambda v\}=\lambda\cdot \conv\{G\cdot v\}$ are the same up to dilation. It is thus useful to consider different types of orbits of rays as in the following definition:


\begin{definition}
A \emph{ray-orbit} is a set of rays $R = \{\ray{x} \ : x\in G\cdot v\}$ for some $v\in\R^3\setminus\{0\}$. A ray-orbit is called \emph{regular}, if $G$ acts regularly on $R$, i.e. $G$ acts transitively with trivial stabilizer on $R$. In this case $|R|=|G|$. If $G$ is not regular on $R$, then $|R|<|G|$ and  we call $R$ \emph{non-regular}.

We further call an 
orbit $R$ a \emph{flip orbit},  if the stabilizer of $R$ consists of exactly the identity and a rotation of order $2$. The respective axis is called a \emph{flip-axis}.
\end{definition}

The following lemma can be used to define a set $\F'$ such that $\F(G)\subset\F'$. 
It turns out (cf. Section \ref{sec:characterization}) that this outer approximation is not far from being exact and for certain groups these conditions are already sufficient to describe $F(G)$. Here, for a set $M\subset \Z^2$ and an integer $n$, we define 
\begin{equation*}
(M \mod n) =\{((x\mod n),(y\mod n)) \mid (x,y)\in M \}\subset (\Z / n\Z)^2.
\end{equation*}

\begin{lemma}\label{lem:f_vectors_mod_n}
Let $O$ denote the set of all non-regular ray-orbits of $G$. Furthermore, let $O_2$ be the set of flip orbits and define $\Onreg:=O\backslash O_2$. For an arbitrary $G$-symmetric polytope $P$ we have
$$(f(P) \mod n) \in \left(  \sum_{X\in \Onreg} \{(|X|,0)\}^\diamond + \sum_{X\in O_2} \{(0,0),(|X|,0)\}^\diamond\mod n\right), $$
where the sum denotes the Minkowski sum of sets. Here we set  the sum over the empty set to be $ \{(0,0)\}$.
\end{lemma}

\begin{example}
Consider the group $G=\D_3$, the dihedral group of order 6. It has one 3-fold rotation axis and three 2-fold rotation axes (flip-axes) perpendicular to the 3-fold rotation axis. 
The orbit of a ray in general position, meaning on none of the rotation axes, has trivial stabilizer and is thus a regular orbit with six elements. Let $r$ be a ray on the $3$-fold rotation axis. All rotations around that axis fix $r$ and any of the three flips map $r$ to $-r$. Hence $\{r,-r\}\subset \Onreg$.   
For a ray $s$  on one of the flip axes, the  stabilizer is generated by the rotation of order 2. The orbit thus is in $O_2$ and has 3 elements. There are two orbits of this kind, namely $D_3 \cdot s$ and $\D_3\cdot (-s)$.

If we now consider a $D_3$- symmetric polytope $P$ and its $f$-vector $f(P)$ modulo $6$, all facets and vertices in general position form orbits of 6 and are thus irrelevant modulo 6.  $P$ can either have a vertex on both $r$ and $-r$ or it can have a facet on both. The facet would then have to be perpendicular to the axis and be invariant under the 3-fold rotation with all its vertices in general position. Independently, $P$ can either have a vertex, an edge or a facet on all three elements in the orbit of $s$ and, again independently, a vertex, an edge or a facet on the orbit of $-s$. Accounting for all independent possibilities leads to the calculation given in Lemma \ref{lem:f_vectors_mod_n}:
\begin{align*}
(F(P) \mod 6) &\in && (  \{(|\{ r, -r \}|,0) \}^\diamond \\
	& &&+ \{(0,0),(|\D_3\cdot s|,0)\}^\diamond + \{(0,0),(|\D_3\cdot (-s)|,0)\}^\diamond  \mod 6 ) \\
	& = && ( \{ (2,0), (0,2)  \} + \{ (0,0), (3,0), (0,3) \} + \{ (0,0), (0,3), (3,0)  \} \\
		& &&\mod 6)\\
	&= && (\{ (0,2), (0,5), (2,3), (3,5)   \}^\diamond \mod 6) 
\end{align*} 
An $f$-vector  $f \equiv (2,3)=(2,0)+(0,3)+(0,0) \mod 6$, for instance, yields that there is a vertex  on both sides of the 3-fold axis, a facet on $D_3\cdot s$  and an edge on $D_3\cdot (-s)$
\end{example}

\begin{proof}[Proof of Lemma \ref{lem:f_vectors_mod_n}]
Let $\ver(P)$ be the set of vertices of $P$ and $\normals(P)$ the set of outer normal vectors of $P$ (thus representing the facets of $P$).
 For a $G$ symmetric set $M\subset\R^3$ we use the notation $$M/G = \{\{\ray{x} \ : \ x\in G\cdot y\} \ : \ y\in M\},$$ the set of all ray-orbits of $M$ under $G$.
By partitioning the vertices (resp. facets) into orbits and summing over the cardinality of these orbits, we get
$$\f(P) = (\sum_{X\in \ver(P)/G} |X|, \sum_{X\in \normals(P)/G} |X|).$$

If $X$ is a regular ray-orbit, then $|X| = n$ and can hence be omitted modulo~$n$:
$$\f(P) \equiv (\sum_{X\in (\ver(P)/G)\cap O} |X|, \sum_{X\in (\normals(P)/G)\cap O} |X|) \mod n,$$
where $O$ is the set of all non-regular ray-orbits of $G$.
Now observe, that each non-regular orbit $X$ intersects either vertices, edges or facets of $P$. If $X\in \Onreg$, then the induced symmetry prevents $X$ from containing edges. 
Therefore, $\Onreg\cap (\ver(P)/G)$ and $\Onreg\cap (\normals(P)/G)$ is in fact a partition of $\Onreg$. Analogously, $O_2\cap (\ver(P)/G)$ and $O_2\cap (\normals(P)/G)$ are disjoint subsets of $O_2$. 
Altogether, that yields
\begin{align*}
f(P)&\equiv (\sum_{X\in (\ver(P)/G) \cap \Onreg} |X|, \sum_{X\in (\normals(P)/G)\cap \Onreg} |X|)\\
&+(\sum_{x\in(\ver(P)/G)\cap O_2} |X|, \sum_{X\in (\normals(P)/G)\cap O_2)} |X| )\\
&\in \sum_{X\in \Onreg} \{(0,|X|),(|X|,0)\} + \sum_{X\in O_2} \{(0,0),(0,|X|),(|X|,0)\}\mod n
\end{align*}
which is equivalent to the assertion.
\end{proof}

\section{Base polytopes}\label{sec:base}

The characterization of $f$-vectors for a given group $G$ mainly consists of two parts. First, we need to find conditions  on $\F(G)$ to show that $\F(G)\subset \F'$ for a given set $\F'$ as the one in Lemma \ref{lem:f_vectors_mod_n} with a few adjustments. Then we need to construct explicit $G$-symmetric polytopes for each $f\in\F'$ to show that $\F'\subset \F(G)$. To do so, we use polytopes with certain properties to construct  infinite families of $G$-symmetric polytopes. These  so-called \emph{base polytopes}, introduced in this section, form the foundations of our constructions.

\begin{definition}
A \emph{base polytope} w.r.t. $G$ is a $G$-symmetric polytope $P$ with the following properties:
\begin{enumerate}
\item $P$ contains a simplicial facet with trivial stabilizer,
\item $P$ contains a simple vertex with trivial stabilizer.
\end{enumerate}
\end{definition}

A general approach to construct a $G$-symmetric polytope from a given symmetric polytope $P$ (for example from a base polytope) is to take a  some vectors $v_1,\dots,v_k$ and consider the convex hull $\conv(P\cup G\cdot\{v_1,\dots,v_k\})$. In order to keep track of how the number of vertices and faces change due to the construction with respect to the faces and vertices of $P$, the following definition is useful:  We say $X$ \emph{sees} $Y$ with respect to $P$, if any of the line segments $\conv(x,y)$ with $x\in X$ and $y\in Y$ does not intersect the interior of $P$.



The next lemma is a technical result ensuring that many operations known for general polytopes can also be used for symmetric polytopes (by adding whole orbits instead of points) without getting unexpected edges and facets.

\begin{lemma} \label{lem:small_disc}
Let $F$ be a face of $P$ with stabilizer $H$ and supporting hyperplane $S=\{x\ : \ a^tx=b\}$ such that $P\subset \{x\ : \ a^tx\leq b\}$. Then there exists an $H$-symmetric disc $D$ contained in a hyperplane $S'=\{x\ : \ a^tx=b'\}$ with $b'>b$, such that $D$ does not see the set $((G\cdot D)\backslash D)\cup (P\backslash F)$. 
Moreover, the center of $D$ is fixed by $H$.
\end{lemma}

\begin{proof}
First, note that $D$ sees the set $((G\cdot D)\backslash D)\cup (P\backslash F)$ if and only if $D$ sees the set $((G\cdot D)\backslash D)\cup (\ver(P)\backslash \ver(Q))$:
 By definition, $D$ cannot see any interior points of $P$. If we take a point $y$ on the boundary of $P$ that is not a vertex, that is, $y$ lies in the relative interior of a face $F'$ of $P$, then a point $x\in D$ sees $y$ if and only if $x$ sees all vertices of $F'$. 

The center point 
$$c = \frac{1}{|\ver(F)|}\sum_{x\in\ver(F)} x \quad \in \relint (F)$$
of the vertices  as well as the outer normal vector $a$ of $F$ are fixed by $H$. Therefore, $c+\delta a$ is also fixed by $H$ for any choice of $\delta\geq 0$. Then, any disc with normal vector $a$ and center on $c+\R_+ a $ is $H$-symmetric.

For $\eps, \delta \geq 0$  define $D(\delta,\eps)$ to be the $H$-symmetric disc with radius $\eps$ parallel to $F$ with distance $\delta$ by  
\begin{equation*}
D(\delta,\eps)\coloneqq \{x \ : \ \| x-p_\delta\| \leq \eps, \ a^tx = a^t p_\delta \},
\end{equation*}  
where $p_\delta:=c+\delta a$. Note that for any vertex $v\in \ver(P)\setminus \ver(F)$ the function $\phi_v$ that sends $(\delta,\eps)$ to the distance between 
$F\cap \conv(D(\delta, \eps), v)$ and the relative boundary of $F$ is continuous and $\phi_v(0,0)>0$. Analogously, for any $A\in G\backslash H$ the function $\phi_A$ that sends $(\delta, \eps)$ to the distance between 
$$F\cap \conv(D(\delta,\eps), A\cdot D(\delta ,\eps))$$ and the relative boundary of $F$ is continuous and also $\phi_A(0,0)>0$. Hence, we find a small $\delta_0>0$ and a small $\eps_0>0$, such that $\phi_v(\delta_0, \eps_0)>0$ and $\phi_A(\delta_0,\eps_0)>0$ for all $v\in\ver(P)\backslash\ver(Q)$ and $A\in G\backslash H$.  Therefore, for 
$D\coloneqq D(\delta_0,\eps_0)$, all line segments $\conv(x,y)$ with $x\in D$, $y\in (G\cdot D)\backslash D\cup P\backslash Q$ intersect the interior of $P$.
By definition, that means that $D$ does not see the set $((G\cdot D)\backslash D)\cup (P\backslash Q)$.
\end{proof}

The most important technique to generate new $G$-symmetric polytopes is by stacking vertices on facets. This is a generalization of the constructions of Steinitz~\cite{steinitz1906polyederrelationen}.


\begin{lemma}\label{lem:stacking}
Let $P$ be a $G$-symmetric polytope and let $F$ be a face of $P$ of degree $k$. Let $H\leq G$ be the stabilizer of $F$.
There is a $G$-symmetric polytope $P^\prime$ with 
$$\f(P^\prime) = \f(P) + \frac{|G|}{|H|}(1,k-1).$$
Furthermore, $P^\prime$ has simplicial facets with trivial stabilizer and a vertex with stabilizer $H$ of degree $k$.

We denote $P'=\CS_{k,|G|/|H|}(P)$ and call the operation \emph{careful stacking on~$F$}.
\end{lemma}

\begin{proof}
We choose $w$ to be the center point of a disc as in Lemma \ref{lem:small_disc}. Set
$$P^\prime = \conv(P\cup Gw).$$
Clearly $P^\prime$ is $G$-symmetric. By the orbit stabilizer theorem, we know that \mbox{$|Gw| = \frac{|G|}{|H|}$}. Furthermore,  we know that all edges incident to $w$ are those between $w$ and the vertices of $F$. Therefore, we can think of $P'$ as the polytope $P$ with pyramids over $F$ and the facets  in the orbit of $F$, resulting in $\frac{|G|}{|H|}$ new vertices ($Gw$) and  $k\cdot \frac{|G|}{|H|}$ new simplicial facets, while the $\frac{|G|}{|H|}$ facets $GF$ are lost.
This yields
$$\f(P^\prime) = \f(P) + \frac{|G|}{|H|}(1,k-1).$$
The facets of the obtained pyramids are simplicial with trivial stabilizer and $w$ has stabilizer $H$ by definition.
\end{proof}

For any operation $\delta$ on polytopes we can define the dual operation that sends $P$ to $(\delta(P^\vee))^\vee)$. Applying that to Lemma \ref{lem:stacking}, we get

\begin{remark}
Let $P$ be a $G$ symmetric polytope and $v$ a vertex of degree $k$ with stabilizer $|H|$.
Then there is a $G$ symmetric polytope $P'$ such that
$$f(P') = f(P) + \frac{|G|}{|H|}(k-1,1)$$
obtained by the dual operation of $\CS_{k,|G|/|H|}$ called \emph{careful cutting} $\CC_{k, |G|/|H|}$. More precisely
$$P' = \CC_{k,|G|/|H|}(P) = (\CS_{k,|G|/|H|}(P^\vee))^\vee.$$
The polytope $P'$ has simple vertices with trivial stabilizer and a facet of degree $k$ with stabilizer $H$.
\end{remark}

Applying the operations $\CS$ and $\CC$ successively to base polytopes, we can generate infinite families of $G$-symmetric polytopes. That shows the existence of 
a whole integer cone of $f$-vectors. Using the notation
$$\Cf~=~(2,1)\N~+~(1,2)\N,$$ we have:

\begin{corollary}\label{cor:base_cone}
Let $P$ be a base polytope with respect to $G$. Then $f(P) + n\Cf\subset \F(G)$. 
\end{corollary}
\begin{proof}
By Lemma \ref{lem:stacking} we know that there is a polytope $P' = \CS_{3,n}(P)$ with $f(P') = f(P) + n(2,1)$. Furthermore, $P'$ has a simple vertex and a simplicial facet with trivial stabilizer. Thus $P'$ is a base polytope. The same is true for $P'' = \CC_{3,n}(P)$. We can thus apply the operations $\CS_{3,n}$ and $\CC_{3,n}$ successively to get
$$f(\CS_{3,n}^a \circ \CC_{3,n}^b (P)) = f(P) + a\cdot (n,2n) + b\cdot (2n,n)$$
for any integers $a,b\geq 0$.
Hence, $f(P) + n\Cf\subset\F(G)$.
\end{proof}

This is the main tool for the construction of polytopes with a given $f$-vector.

\section{Certificates}\label{sec:certificates}

Since it is often impossible to construct symmetric base polytopes with small $f$-vector entries, we show that it is possible to replace one base polytope by several polytopes with certain weaker properties and still get an integer cone of $f$-vectors as in Corollary \ref{cor:base_cone}. We thus introduce the concept of certificates, a collection of one or more polytopes with certain properties that ensure (certify) the existence of an integer cone of f-vectors of the form $v+nf$ for some vector $v\in \N^2$ and $n=|G|$. The weaker properties on polytopes we introduce are called left and right type, where the name is only due to our choice of writing down the certificate. 

\begin{definition}\label{def:left_right_type}
Let $f\in \Z^2$. 
A \emph{\rt} (\emph{\lt}) polytope w.r.t. $G$ and $f$ is a $G$-symmetric polytope $P$ with $f(P) = f$ which has simple vertices (simplicial facets) with trivial stabilizer.
If it is understood in the context, we omit the group $G$.
\end{definition}

Note that a {\lt} polytope $P$ can be used to construct a base polytope $P'$ with $f(P') = f(P) + (n,2n)$ by a single $\CS_{3,n}$ operation. On the other hand, a {\rt} polytope $P$ can be used to construct a base polytope $P'$ with $f(P') = f(P) + (2n,n)$ by a single $\CC_{3,n}$ operation. The  {\lt} and {\rt} polytopes can thus be interpreted as 'half base'.

\begin{definition}
An \emph{RL-certificate} for a vector $\f$ consists of a {\rt} polytope $P_R$ w.r.t. $\f$ and a {\lt} polytope $P_L$ w.r.t. $\f + (n,2n)$. It can be visualized as: 
\begin{center} \defLR {$P_L$} {$P_R$} \certRL \end{center}
An \emph{LR-certificate} for a vector $\f$ consists of a {\lt} polytope $P_L$ w.r.t. $\f$ and a {\rt} polytope $P_R$ w.r.t. $\f + (2n,n)$. It can be visualized as:
\begin{center}\defLR {$P_L$} {$P_R$} \certLR\end{center}
A \emph{triangle-certificate} for a vector $\f$ consists of four polytopes $P_L, P_R, P, Q$ such that
\begin{enumerate}
\item $P_L$ is a {\lt} polytope w.r.t. $\f + (n,2n)$,
\item $P_R$ is a {\rt} polytope w.r.t. $\f + (2n,n)$,
\item $P$ is some $G$-symmetric polytope with $\f(P) = \f$ and
\item $Q$ is some $G$-symmetric polytope with $\f(Q) = \f + (3n, 3n)$.
\end{enumerate}
It can be visualized as:
\begin{center}\defLRTX {$P_L$} {$P_R$} {$P$} {$Q$} \certTri\end{center}
A \emph{B-certificate} for a vector $\f$ consists of a base  type polytope $P_B$ with $\f(P_B) = \f$ or two polytopes $P_L, P_R$ such that $P_L$ is a {\lt} polytope and $P_R$ is a {\rt} polytope, both with respect to $\f$. It can be  visualized as:
\begin{center}\defT {$P_B$} \begin{adjustbox}{valign=c}\certBtikz\end{adjustbox}
\quad or
\defT {$P_L,P_R$} \begin{adjustbox}{valign=c}\certBtikz\end{adjustbox}.\end{center}
We say that we \emph{have a certificate} for a vector $\f$ if there is either an LR-certificate, an RL-certificate, a triangle-certificate or a base-certificate w.r.t. $\f$.
\end{definition}

The benefit of certificates is the following theorem:

\begin{theorem}\label{thm:certificates}
Let $\f\in\N^2$.
If we have a certificate for $\f$ then $\f + n\Cf \subset \F(G)$.
\end{theorem}
\begin{proof}
No matter the type of the certificate, we can use Corollary \ref{cor:base_cone} and Definition \ref{def:left_right_type} to construct $G$-symmetric polytopes with $\f$-vectors $\f, \f + (n,2n), \f + (2n,n)$ and $\f + (3n,3n)$ as well as two $G$-symmetric base polytopes $P$ and $Q$ with $\f(P) = \f + (4n,2n)$ and $\f(Q) = \f + (2n,4n)$. Consider any vector $v = \f + a\cdot (n,2n) + b\cdot (2n,n) \in \f + n\Cf$. If $a+b\leq 2$ then $v$ is an $\f$-vector of one of the polytopes given above. If $a+b \geq 3$ then either $a$ or $b$ are bigger than two. Thus, if $a\geq 2$, we have
$$ v = \f + (2n,4n) + (a-2)\cdot (n,2n) + b\cdot (2n,n)$$
and a $G$-symmetric polytope can be constructed via Corollary \ref{cor:base_cone} using $a-2$ careful stacking operations and $b$ careful cutting operations on $Q$.
When $b\geq 2$ we can use an analog argument on $P$.
\end{proof}


The following Lemma shows, for which vectors $f\in \N^2$ we need a certificate to obtain all $f$-vectors given by a certain congruence relation.

\begin{lemma}\label{lem:finite_cone_distribution}
Let $p,q,n$ be integers such that $0\leq p\leq q < n$. Furthermore, let $m$ be the smallest integer such that $m\geq 4+p-2q$ and $n|m$.
Moreover, let 
$$X(p,q) = \{\f\in\F \ : \ \f\equiv (p,q)\mod n\}.$$
Then
$$X(p,q) = (p,q) + \{v_1,v_2,v_3\} + n\Cf.$$
Where
\begin{align*}
v_1(p,q) =& \begin{cases}(m,m) \text{ if } m+2p-q\geq 4\\ (m+2n, m+n) \text{ otherwise}\end{cases}\\
v_2(p,q) =& \begin{cases}(m+n, m+n) \text{ if } m + n + 2p-q \geq 4\\ (m+3n, m+2n) \text{ otherwise}\end{cases}\\
v_3(p,q) =& \begin{cases}(m+2n, m+2n) \text{ if } m+2n+2p-q \geq 4 \\ (m+4n, m+3n) \text{ otherwise}\end{cases}.
\end{align*}
\end{lemma}

In other words, Lemma \ref{lem:finite_cone_distribution} shows that if we have certificates for $f=(p,q)+v_i$, $i\in {1,2,3}$, we have shown that $X(p,q)\subset F(G)$.  
\begin{proof}
It is obvious that $\F$ is a union of translates of the fundamental domain of the lattice $\Lambda$ generated by $n\cdot(2,1)$ and $n\cdot(1,2)$ starting with
$X \coloneqq \{(a,b) : 4\leq 2a-b, 2b-a < 4 + 3n\}$ (the translate in $(4,4)$).
There are exactly three points of $X(p,q)$ in $X$ since $\Lambda$ has determinant $3n^2$ while the lattice $n\Z^2$ has determinant $n^2$.
It is easy to check, that these are exactly the points $v_1, v_2, v_3$ given above.
\end{proof}

With Theorem \ref{thm:certificates} and Lemma \ref{lem:finite_cone_distribution} as well as the fact that $\F(G)^\diamond = \F(G)$, it easily follows:

\begin{corollary} \label{cor:certificates_work}
Let $0\leq p_i\leq q_i <n$ for $i=1,\dots,r$ and 
$$\F' = \{f\in\F \ : \ f\equiv (p_i,q_i) \text{ for some } i=1,\dots,r\}^\diamond.$$
If, for every $i=1,\dots,r$ and every $k=1, 2 ,3$ we have a certificate with respect to the group $G$ and the vector $(p_i,q_i) + v_k(p_i,q_i)$ (as in Lemma \ref{lem:finite_cone_distribution}), then $\F'\subset\F(G)$.
\end{corollary}

Consequently, to show that $\F(G)$ contains a set $\F'$ of the form as in  Lemma \ref{lem:f_vectors_mod_n} it suffices to state a list of certificates.

\section{Constructions of symmetric polytopes}\label{sec:constructions}

It is easy to construct polytopes symmetric under $G$ by taking the convex hull of several orbits. However, in most cases it is hard to control the $f$-vector of the resulting polytope. 
In this section, we describe some special operations that can be applied to $G$-symmetric polytopes to obtain other $G$-symmetric polytopes. These operations will be used in Section \ref{sec:characterization} to construct symmetric polytopes needed for certificates. We will emphasize the implications for the $f$-vector and the types of the polytope. The first operation we discuss is to stack smaller copies of a facet of a polytope on that facet to get a kind of narrowed prism.
\begin{figure}[h]
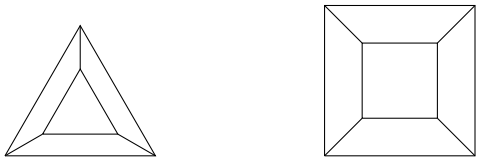

\RPpic
\caption{Projection of the regular prism operation over a $3$-gon and a $4$-gon}\label{fig:RP}
\end{figure}

\begin{lemma}[$k$-gon prism]\label{lem:RP}
Let $F$ be a face of $P$ with degree $k$. Let $H$ be the stabilizer of $F$. Then there exists a $G$-symmetric polytope $P'$ with
\begin{equation*}
f(P') = f(P) + \frac{|G|}{|H|}(k,k).
\end{equation*}
Furthermore, $P'$ is {\rt} and has a facet of degree $k$ with stabilizer $H$.
If $P$ is {\lt} then $P'$ is base.
We denote $P' = \RP_{k,|G|/|H|}(P)$ and call it \emph{regular prisms} or \emph{prisms} over $F$.
\end{lemma}
\begin{proof}
We choose $F'$ to be a small copy of $F$ contained in a small disc over $F$ chosen as in Lemma \ref{lem:small_disc}. Define $P' = \conv(P\cup G\cdot F')$. $P'$ has $k$ additional vertices over $F$, namely the vertices of $F'$. Furthermore, $P'$ contains $k+1$ additional facets over $F$, one quadrilateral facet between any edge of $F$ and its counterpart in $F'$ as well as $F'$ itself, while $F$ is no face of $P'$.  Since the same argument holds for all the $\frac{|G|}{|H|}$ elements in the orbit of $F$, we get 
$$f(P') = f(P) + \frac{|G|}{|H|}(k,k)$$
as desired. The vertices of $F'$ are simple with trivial stabilizer, thus $P'$ is {\rt}. If $P$ has a simplicial facet with trivial stabilizer which is not contained in the orbit of $F$, then this facet is also a facet of $P'$. If, on the other hand, $F$ is a simplicial facet with trivial stabilizer, then so is $F'$. In any case, if $P$ is {\lt} then $P'$ is also {\lt} and thus base.
\end{proof}

By Lemma \ref{lem:small_disc} we can apply any operation over a two dimensional polytope to a full orbit of facets without unexpected edges and facets. In the following we state a few such operations. The missing proofs are all analougusly to the proof of Lemma \ref{lem:RP}.

For example, we can also stack a smaller rotated copy of a facet over itself to get the following:

\begin{figure}[h]
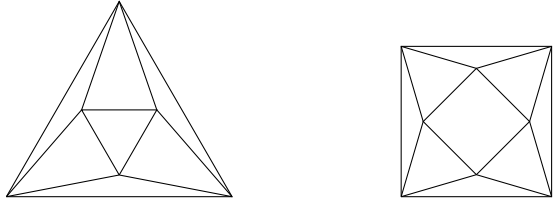

\TPpic
\caption{Projection of the twisted prism operation over a $3$-gon and a $4$-gon}
\label{fig:TP}
\end{figure}

\begin{lemma}[Twisted prism]
Let $\F$ be a face of $P$ with $k\geq 3$ vertices and let $H$ be the stabilizer of $F$. Then there exists a $G$-symmetric polytope $P'$ with symmetry group $G$ and
\begin{equation*}
f(P')=f(P)+\frac{|G|}{|H|}(k,2k).
\end{equation*}
Additionally, $P'$ has simplicial facets with trivial stabilizer and is thus {\lt}.
We denote $P'=\TP_{k,|G|/|H|}(P)$ and call it \emph{twisted prisms over $F$}.
\end{lemma}
%

Instead of a small copy of itself, one can also stack other facets over a given facet. It can yield interesting transitions of the $f$-vector when the appended faces are not in a general position, lets say some edges are parallel to the edges of the given face. The following Lemma illustrates the example of stacking a $2k$-gon regularly on a $k$-gon.

\begin{figure}[h]
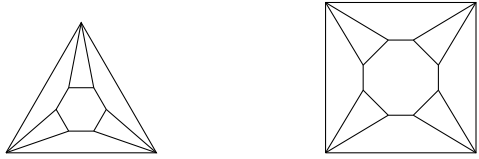

\BPpic
\caption{Orthogonal projection of the big prism operation over a $3$-gon and a $4$-gon}\label{fig:BP}
\end{figure}

\begin{lemma}[2k-gon on k-gon]
Let $F$ be a face of $P$ of degree $k$ and $H$ be the stabilizer of $F$. 
Then there exists a $G$-symmetric polytope $P'$ with
$$f(P') = f(P) + \frac{|G|}{|H|}(2k,2k).$$
Furthermore, $P'$ is base and has a facet of degree $2k$ with stabilizer $H$.
We denote $P' = \BP_{2k, |G|/|H|}(P)$ and call it \emph{big prisms} over $F$.
\end{lemma}

Similar to the above construction we may also stack a $k$-gon on a $2k$-gon in a regular manner if the stabilizer allows to do so.

\begin{figure}[h]
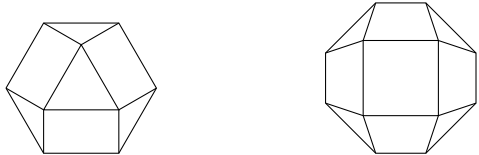

\HPpic
\caption{Projection of the half prism operation over a $6$-gon and an $8$-gon}\label{fig:HP}
\end{figure}

\begin{lemma}[$k$-gon on $2k$-gon] \label{lem:HP}
Let $F$ be a face of $P$ with $2k$ vertices for some $k\in \Z_{\geq3}$. Let $H$ be the stabilizer of $F$ and suppose that $|H|$ divides $k$ and $H$ does not contain any reflections. Then there exists a $G$-symmetric polytope $P'$ with
\begin{equation*}
f(P')=f(P)+\frac{|G|}{|H|}(k,2k).
\end{equation*}
Furthermore, $P'$ is {\lt}.

We denote $P'=\HP_{2k,|G|/|H|}(P)$ and call it \emph{half prisms over $F$}.
\end{lemma}

The constructions seen before share a strong 'regularity'. The stabilizer of a facet implies a symmetry for all facets we stack over it. So it is a reasonable question if we are able to stack something without such 'regularity' (Say, there are no parallel edges). At least if the stabilizer only contains rotations, this is the case. This is illustrated by the following result.

\begin{figure}[h]
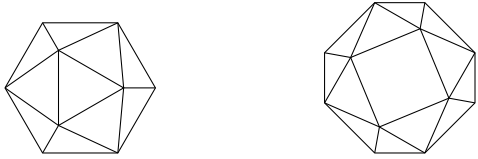

\THPpic
\caption{Projection of the twisted half prism operation over a $6$-gon and an $8$-gon}\label{fig:THP}
\end{figure}
\begin{remark}
Under the conditions of Lemma~\ref{lem:HP}, there exists a $G$-symmetric polytope $P'$ with
\begin{align*}
f(P')=f(P)+\frac{|G|}{|H|}(k,3k).
\end{align*}
Furthermore, $P'$ is {\lt}.

We denote $P'=\THP_{2k,|G|/|H|}(P)$ and call it \emph{twisted half prisms over $F$}. This can be achieved by slightly rotating $F''$ in the proof of Lemma~\ref{lem:HP}, such that the edges of $F''$ are not parallel to the ones of $F$.
\end{remark}



We summarize the constructions given in this section in the following list:

\begin{corollary}\label{cor:operations}
For a $G$ symmetric polytope $P$, we have the following operations on $P$ ($P'$ denotes the polytope obtained by the operation):

\hspace*{-16pt}
\begin{tabular}{llll}
name & conditions on $P$ & type &  $f(P')-f(P)$\\
$\CS_{k,m}$ & $\exists$ facet $F$, $\deg(F) =k$, $|G_F| = n/m$ & {\lt} & $m(1,k-1)$\\
$\CC_{k,m}$ & $\exists$ vertex $v$, $\deg(v) = k$, $|G_v| = n/m$ & {\rt} & $m(k-1,1)$\\
$\RP_{k,m}$ & $\exists$ facet $F$, $\deg(F) = k$, $G_F = n/m$ & {\rt} & $m(k,k)$\\
$\RP^\vee_{k,m}$ & $\exists$ vertex $v$, $\deg(v) = k$, $G_v = n/m$ & {\lt} & $m(k,k)$\\
$\TP_{k,m}$ & $\exists$ facet $F$, $\deg(F) = k$, $|G_F| = n/m$ & {\lt} & $m(k,2k)$\\
$\TP^\vee_{k,m}$ & $\exists$ vertex $v$, $\deg(v) = k$, $|G_v| = n/m$& {\rt} & $m(2k,k)$\\
$\HP_{2k,m}$ & $\exists$ facet $F$, $\deg(F) = 2k$, $|G_F| = n/m$ & {\lt} & $m(k,2k)$\\
$\THP_{2k,m}$ & $\exists$ facet $F$, $\deg(F) = 2k$, $|G_F| = n/m$ & {\lt} & $m(k,3k)$\\
$\BP_{k,m}$ & $\exists$ facet $F$, $\deg(F) = k/2$, $|G_F| = n/m$ & {\bt} & $m(k,k)$\\
\end{tabular}
\end{corollary}
\begin{proof}
This Corollary is basically a summary of this section.
Note that for every operation $\phi$ of the above with certain conditions on $P$, there is a dual operation defined by $\phi^\vee(P) = (\phi(P^\vee))^\vee$. Clearly, the conditions on the argument and the properties of the image are dual to the conditions and properties according to $\phi$.
\end{proof}


For the construction of polytopes via the above described methods, we need explicit examples of polytopes to start with. These polytopes need to have  'small' $f$-vectors, since the constructions can only increase the number of vertices and facets. In the following, we give a list of well known polytopes that are used in this work.

\begin{notation}\label{not:special_polytopes} The following polytopes are used in this paper without further discussion. They are all part of the Platonic, the Archimedean and the Catalan solids and thus well studied in literature. 
Note that this is not a complete list.

\hspace*{-16pt}
\begin{longtable}{llcl}
 \textit{abbr.}  & \textit{name} &  \textit{rotat. symm.}   & \textit{f-vector} \\
 $Tet$ & tetrahedron & $ \T$ & $(4,4)$\\
$TrTet$ & truncated tetrahedron & $ \T$  & $(12,8)$\\
$Oc$ & octahedron & $ \O$ & $(6,8)$\\
 $Cub$ &  cube   & $  \O$ & $(8,6)$\\
$CubOc$ & cuboctahedron & $ \O$ & $ (12,14)$ \\ 
$RDo$ &  rhombic dodecahedron & $ \O$ & $ (14,12)$ \\
$TrCub$ & truncated cube & $ \O$ & $ (24,14)$ \\
 $RCubOc$ & rhombicuboctahedron & $\O$ & $ (24, 26)$ \\
 $SnCub$ & snub cube & $ \O$ & $ (24,38)$\\
 $TrCubOc$ & truncated cuboctahedron & $ \O$ & $ (48,26)$ \\
 $Ico$ & icosahedron & $ \I$ & $ (12,20)$\\
 $ID$ & icosidodecahedron & $ \I$ & $ (30,32)$\\
 $TrI$ & truncated icosahedron & $ \I$ & $ (60,32)$\\
 $RID$ & rhombicosidodecahedron & $ \I$ & $ (60,62)$\\
 $SnDo$ & snub dodecahedron & $ \I$ & $ (60,92)$\\
 $TrID$ & truncated icosidodecahedron & $ \I$ & $ (120,62)$\\

\end{longtable}


\end{notation}

\section{Characterization of $f$-vectors}\label{sec:characterization}

In this section, we go through all finite orthogonal rotation and rotary reflection groups, as described in Theorem \ref{thm:finite_orthogonal_groups}, and characterize their $f$-vectors using the tools developed in the previous sections.

We start with the group $C_n$, the cyclic group of order $n$, which is generated by a rotation with rotation-angle $2\pi/n$ around a given axis. Thus, $C_n$ has two non-regular orbits of size $1$, namely the two rays of the rotation axis.
These are flip-orbits if and only if $n=2$.

\begin{theorem} \label{thm:C_n}
For $n>2$ we have
\begin{align*}\F(\C_n) &= \{\f\in\F \ : \ \f\equiv (1,1) \mod n\}^\diamond\\
&\cup \{\f = (\f_0,\f_2) \in \F \ : \ \f\equiv (0,2) \mod n, 2\f_0 - \f_2 \geq 2n-2\}^\diamond.
\end{align*}
\end{theorem}

\begin{figure}
\begin{center}
\includegraphics[scale=0.2]{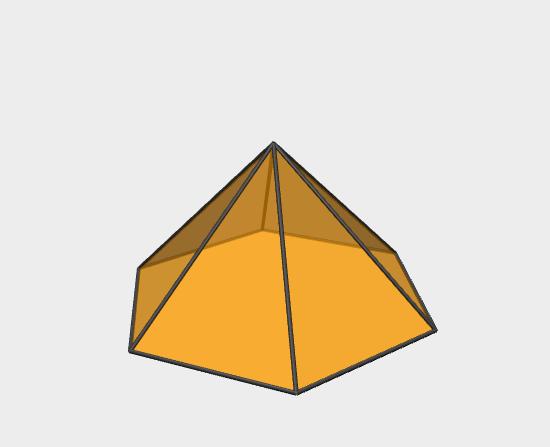}
\includegraphics[scale=0.2]{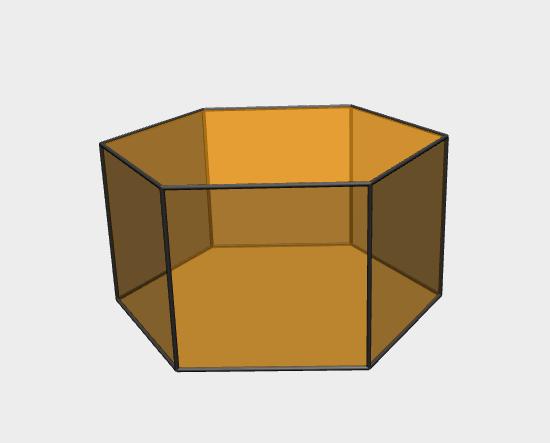}
\includegraphics[scale=0.2]{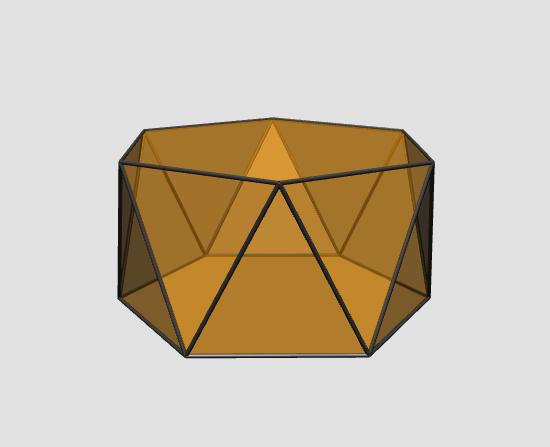}
\end{center}
\caption{The $\C_6$-symmetric polytopes $Pyr_6$, $Pri_6$ and $TPri_6$. This figure as well as following similiar figures are created using Sage \cite{SageMath} and Polymake \cite{polymake:2000}}\label{fig:special_cyc}
\end{figure}

In the proof, we need the following special $\C_n$-symmetric polytopes, see Figure~\ref{fig:special_cyc} for images:
\begin{enumerate}

\item $Pyr_k$ is a pyramid over a regular $k$-gon. $f(Pyr_k) = (k+1,k+1)$.

\item $Pri_k$ is a prism over a regular $k$-gon. $f(Pri_k) = (2k, k+2)$.

\item $TPri_k$ is a twisted prism over a regular $k$-gon. $f(TPri_k) = (2k, 2k+2)$.
\end{enumerate}

\begin{proof}
Denote by $\F'$ the right hand side of the assertion.
By Lemma \ref{lem:f_vectors_mod_n}, we have for any $\C_n$-symmetric polytope $P$
$$\F(P) \subset \{(1,0)\}^\diamond + \{(1,0)\}^\diamond \mod n\equiv \{(0,2),(1,1)\}^\diamond \mod n.$$
The case $f(P)\equiv (0,2) \mod n$ yields that $P$ has one facet with at least $n$ vertices on each non-regular orbit. Counting $\{12\}$ -flags we therefore have
$$2f_1(P) = f_{12}(P) \geq 3\cdot (f_2(P) - 2) + 2\cdot n.$$
This is, by the Euler Equation \eqref{thm:Euler_Steinitz} \eqref{fml:Euler}, equivalent to
$2\f_0(P) - \f_2(P) \geq 2n-2.$
Since $\F(\C_n)$ is invariant under $\diamond$, we thus know that $\F(\C_n)\subset\F'$. Consider Corollary \ref{cor:certificates_work} and the following table to see that $F'\subset \F(\C_n)$:

\tcs
\begin{tabular}{cc|ccc}
$(p,q)$ & \ \hspace{0.7cm} & $v_1(p,q)$ & $v_2(p,q)$ &$v_3(p,q)$\\
\hline
$(1,1)$&
& \defRoot{$(n+1,n+1)$}\defT {$Pyr_n$} \certB
& \defRoot{$(2n+1,2n+1)$}\defT {$Pyr_{2n}$} \certB
& \defRoot{$(3n+1,3n+1)$}\defT {$Pyr_{3n}$} \certB\\
\hline
$(0,2)$&
& \defRoot{$(2n,n+2)$}\defLR {$\TP_{n,1}(*)$} {$Pri_n$}\certRL
& \defRoot{$(3n,2n+2)$}\defLR {$\TP_{n,1}(*)$} {$\RP_{n,1}(Pri_n)$}\certRL
& \defRoot{$(2n,2n+2)$}\defLR {$TPri_n$} {$\RP_{n,1}^2(Pri_n)$}\certLR
\end{tabular}

\end{proof}


Unlike $C_n$ with $n>2$, the group $\C_2$ has two flip-orbits. Thus, any $\C_2$-symmetric polytope may have edges with non-trivial stabilizer. This gives us tremendously more freedom in the construction of $\C_2$ symmetric polytopes.

We consider the group $\C_2$ as rotations around the z-axis. 
The next result shows that, in fact, any $f$-vector can be realized by a $\C_2$ symmetric polytope.

\begin{theorem}
We have
\begin{align*}
\F(\C_2) = \F.
\end{align*}
\end{theorem}

\begin{figure}[h]
\includegraphics[scale=0.15]{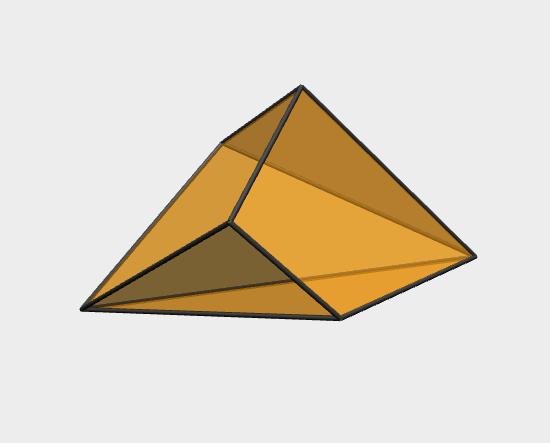}
\includegraphics[scale=0.15]{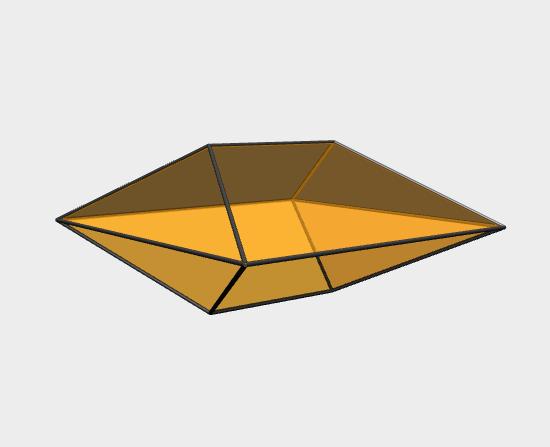}
\includegraphics[scale=0.15]{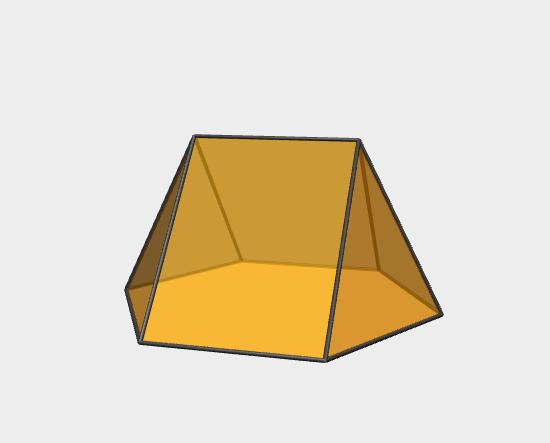}
\includegraphics[scale=0.15]{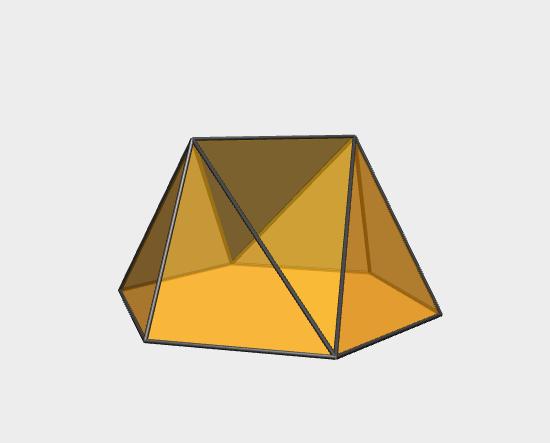}
\caption{The $\C_2$-symmetric polytopes $ST$, $DT$, $RT_6$ and $TT_6$}\label{fig:special_cyc_two}
\end{figure}


In the proof we will need the following special $\C_2$ symmetric polytopes with small $f$-vectors. See Figure~\ref{fig:special_cyc_two} for images.

\begin{itemize}

\item $ST := \conv \left(\C_2\cdot\{(-1,0,1),(1,1,0),(1,-2,-1)\}\right)$ is a polytope that looks like a 'scattered tent'. $f(ST) = (6,6)$.

\item $DT := \conv \left(\C_2\cdot\{(1,0,1), (2,2,0), (2,-2,0), (1,0,-1)\}\right)$ is a polytope that looks like a tent above and below a $4$-gon. $\f(DT) = (8,8)$.

\item $RT_{2k}$ (regular tent) is constructed by taking a regular $2k$-gon and stack a small edge above it, such that the new edge is paralell to two edges of the $2k$-gon. $f(RT_k) = (k+2,k+1)$.

\item $TT_k$ (twisted tent over $k$-gon) can be constructed by taking a regular $k$-gon and stack a small edge above it such that the new edge is not paralell to any edge of the $k$-gon. $f(TT_k) = (k+2,k+3)$.

\end{itemize}

\begin{proof}
Of course $\F(\C_2)\subset \F$. To see that 
$$\F = \{f\in \F \ : \ f\equiv (0,0),(0,1),(1,1) \mod 2 \}^\diamond\subset F(\C_2)$$
consider Corollary \ref{cor:certificates_work} and the following table:

\begin{longtable}{cc|ccc}
$(p,q)$  & \hspace{0.0cm} & $v_1(p,q)$ & $v_2(p,q)$ & $v_3(p,q)$\\
\hline
$(0,0)$&
& \defRoot{$(4,4)$}\defT {$Tet$} \certB
& \defRoot{$(6,6)$}\defT {$ST$} \certB
& \defRoot{$(8,8)$}\defT {$DT$} \certB\\
\hline
$(0,1)$ &
& \defRoot{$(6,5)$}\defT {$RT_4$} \certB
& \defRoot{$(8,7)$}\defT {$RT_6$} \certB
& \defRoot{$(6,7)$}\defT {$TT_4$} \certB\\
\hline
$(1,1)$ &
& \defRoot{$(5,5)$}\defT {$Pyr_4$} \certB
& \defRoot{$(7,7)$}\defT {$Pyr_6$} \certB
& \defRoot{$(9,9)$}\defT {$Pyr_8$} \certB\\
\end{longtable}
\end{proof}

Next, we characterize the $f$-vectors for the group $\Dih_d$.
This group consists of a $d$-fold rotation around a given axis $v$ and $2d$ two-fold rotations around axes that are orthogonal to $v$.
So we have one non-regular orbit consisting of the two rays belonging to $v$ which is a flip-orbit if and only if $d=2$.
Furthermore, $\Dih_d$ has two flip orbits of size $d$, whose rays each make up half of the two-fold rotation axes.


\begin{figure}
\begin{center}
\includegraphics[scale=0.2]{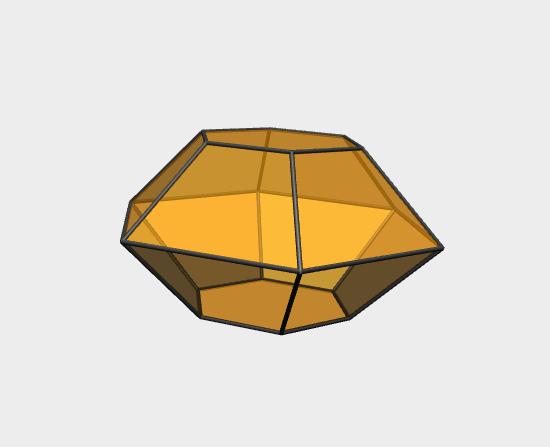}  
\includegraphics[scale=0.2]{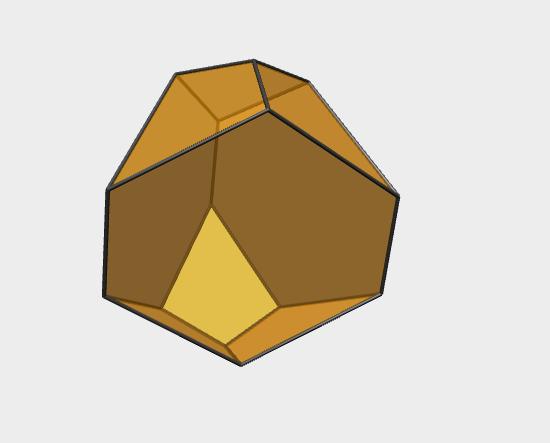}
\includegraphics[scale=0.2]{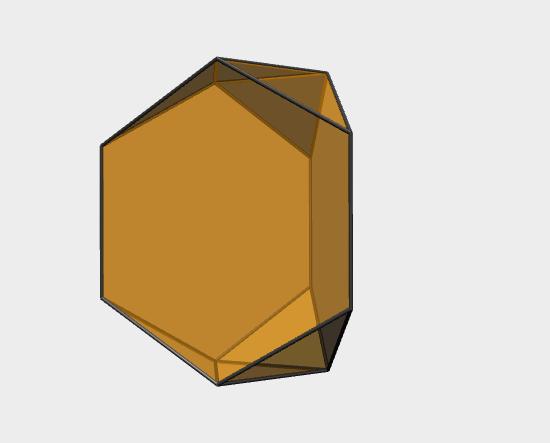}  
\includegraphics[scale=0.2]{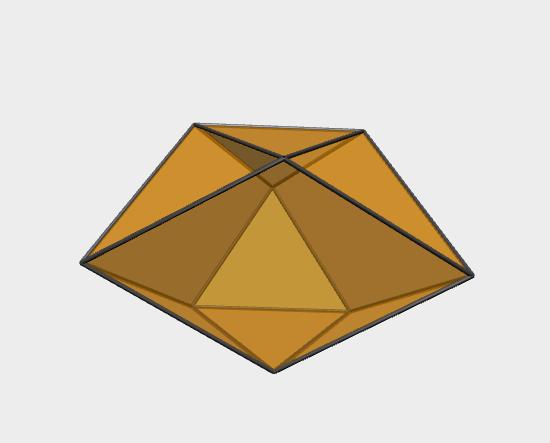}
\includegraphics[scale=0.2]{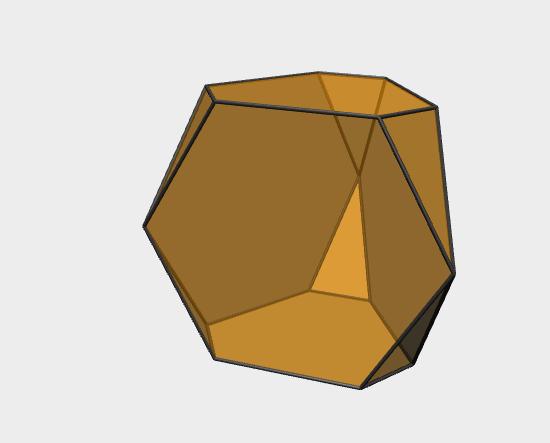}
\end{center}
\caption{The $D_5$-symmetric polytopes $DPri_6$, $Dia_3$, $EB_{6,3}$, $B_{4,3}$ and $B_{6,3}$}\label{fig:special_dih}
\end{figure}

These observations are sufficient to characterize $f$-vectors for the group $\Dih_d$:

\begin{theorem}
For $d > 2$ and $n=2d$ we have
\begin{align*}
\F(\Dih_d) &= \{\f\in\F \ : \ \f\equiv (0,2),(2,d) \mod n\}^\diamond\\
&\cup \{\f = (\f_0,\f_2)\in\F \ : \f\equiv (0,d+2), (d,d+2) \mod n , 2\f_0 - \f_2 \geq 3d-2\}^\diamond\\
\end{align*}
\end{theorem}

For the construction of small $f$-vectors we need the following polytopes, which are $\Dih_d$ symmetric for certain parameters (see also Figure~\ref{fig:special_dih}):
\begin{itemize}
\item $DPri_k$ (double prism on a $k$-gon), a regular $k$-gon with a smaller copy above and below. $f(DPri_k)= (3k, 2k+2)$. This is $\Dih_d$ symmetric when $d$ divides $k$.
 \item $Dia_k$ (diamond of order $k$). This is the dual of the following polytope: A regular $k$-gon with twisted smaller copies above and below. \\
$f(Dia_k)=(4k+2,3k)$. This is $\Dih_d$ symmetric when $d$ divides $k$. 
\item $EB_{2k,l}$ (edge belt): the convex hull of $l$ regular $2k$-gons arranged along an $l$-gon, such that two neighboring $k$-gons intersect in an edge. \\
 $f(EB_{k,l})=(l\cdot 2(k-1), 2 \cdot \left\lfloor \frac{k-1}{2} \right\rfloor\cdot l +l+2)$. This is $\Dih_d$ symmetric when $d$ divides $l$.

\item $B_{2k,l}$ (belt): the convex hull of $l$ regular $2k$-gons arranged along an $l$-gon, such that two neighboring $2k$-gons intersect in a vertex. \\ $f(B_{2k,l}) = ((2k-1)\cdot l, 2(1+\lfloor \frac k 2 \rfloor)\cdot l +2)$. This is $\Dih_d$ symmetric, when $d$ divides $l$.

\end{itemize}

\begin{proof}
Denote by $\F'$ the right hand side of the assertion.
By Lemma \ref{lem:f_vectors_mod_n} we know that
\begin{align*}
\F(\Dih_d) &\subset \{f\in \F \ : \ (f\mod n)\in \{(2,0)\}^\diamond + \{(0,0),(0,d)\}^\diamond + \{(0,0),(0,d)\}^\diamond\}\\
&= \{f\in \F \ : \ f\equiv (0,2),(0,d+2),(2,d),(d,d+2)\mod n\}^\diamond.
\end{align*}
If $\f_2(P) \equiv d+2 \mod n$ then $P$ has facets on the rotation axis containing at least $d$ vertices.
Furthermore, $P$ has also facets on exactly one of the flip-orbits, containing at least $4$ vertices. By counting $\{1,2\}$-flags we have
$$2f_1(P) = f_{12}(P) \geq 3\cdot(f_2-d-2) + d\cdot 2 + 4\cdot d.$$
By applying Euler's equation \ref{thm:Euler_Steinitz} \eqref{fml:Euler}, we have
$2\f_0(P) - \f_2(P) \geq 3d-2$.
Since $\F(\Dih_d)$ is invariant under $\diamond$, we know that $\F(\Dih_d)\subset \F'$. To see that $\F'\subset \F(\Dih_d)$ consider Corollary \ref{cor:certificates_work} and the following table:

\tcs
\begin{longtable}{cc|ccc}
$(p,q)$  &  \ \hspace{0.7cm} & $v_1(p,q)$ & $v_2(p,q)$ & $v_3(p,q)$\\
\hline
$(0,2)$ &
& \defRoot{$(2n,n+2)$}\defLR {$\BP_{d,2}(TPri_d)$} {$Pri_n$} \certRL
& \defRoot{$(n,n+2)$}\defLR {$TPri_d$} {$DPri_n$} \certLR
& \defRoot{$(2n,2n+2)$}\defT {$\RP_{d,2}(TPri_d)$} \certB\\
\hline
$(2,d)$&
& \hspace{-14pt} \defRoot{$(2n+2,n+d)$}\defLR {$\CS_{6,2}\circ\CC_{3,n}(Pri_d)$} {$Dia_d$}  \certRL
& \defRoot{$(n+2, n+d)$}\defLR {$\CS_{d,2}(Pri_d)$} {$\RP^\vee_{d, 2}(Dia_d)$} \certLR
&  \hspace{-14pt} \defRoot{$(2n+2, 2n+d)$}\defLR {$\CS_{d,2}\circ\RP_{d,2}(Pri_d)$} {$(\RP^{\vee}_{d, 2})^2(Dia_d)$} \certLR\\
\hline
$(0,d+2)$&
&  
\defRoot{$(n,d+2)$}\defLR {$\TP_{d,2}(*)$} {$Pri_d$} \certRL
& \defRoot{$(2n,n+d+2)$}\defT {$EB_{6,d}$} \certB
& \defRoot{$(3n,2n+d+2)$}\defT {$\BP_{d,2}(Pri_d)$} \certB\\
\hline
$(d,d+2)$&
&  \hspace{-14pt} 
\defRoot{$(2n+d,n+d+2)$}\defT {$B_{6,d}$} \certB
& \defRoot{$(n+d,n+d+2)$}\defLR {$B_{4,d}$} {$B_{8,d}$} \certLR
& \defRoot{$(2n+d,2n+d+2)$}\defT {$\RP_{d,2}(B_{4,d})$} \certB\\
\end{longtable}
\end{proof}

Next, we consider $\Dih_2$. This group can be interpreted as the group of all flips around coordinate axes. Thus, $\Dih_2$ has exactly three flip-orbits of size $2$, each consisting of the rays of one flip-axis.

Interestingly, the respective characterization of $f$-vectors contains the exceptional case $f = (6,6)$, which can not be realized by a $\Dih_2$ symmetric polytope.


\begin{theorem}
We have
\begin{align*}
\F(\Dih_2) &= \{\f\in\F \ : \ \f\equiv (0,0),(0,2),(2,2) \mod 4\}^\diamond\setminus\{(6,6)\}.
\end{align*}
\end{theorem}

For small $f$-vectors we consider the following special $\Dih_2$-symmetric polytopes (see Figure~\ref{fig:special_dih_two}):
\begin{itemize}


\item $Dih_2(10,10) = \conv(\Dih_2\cdot \{(6,0,0),(1,1,1),(2,1,-2)\}$

\item $Dih_2(10,14) = \conv(\Dih_2\cdot \{(2,0,0),(1,1,1),(\frac 34, 1, \frac 32)\})$

\item $Dih_2(12,12) = \conv \{\Dih_2\cdot (2,1,1), (-2,1,1), (1,0,2)\}$

\item $Dih_2(14,14) = \conv \{\Dih_2\cdot (4,0,0),(1,3,0),(2,0,1),(0,2,1)\}$

\item $Dih_2(18,18) = \conv \{\Dih_2\cdot (4,0,0),(1,3,0),(2,0,1),(0,2,1),(\frac 18, \frac{23}{8}, 1)\}$
\end{itemize}

\begin{figure}
\begin{center}
\includegraphics[scale=0.2]{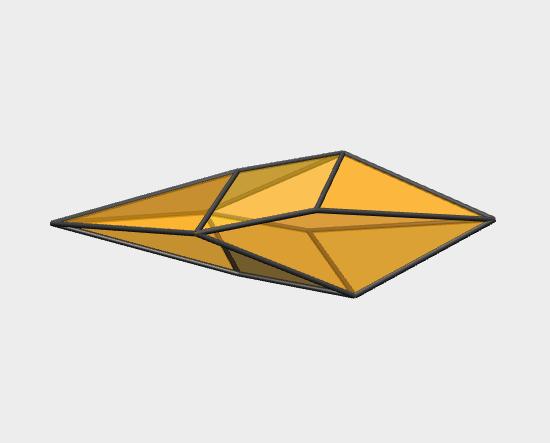}
\includegraphics[scale=0.2]{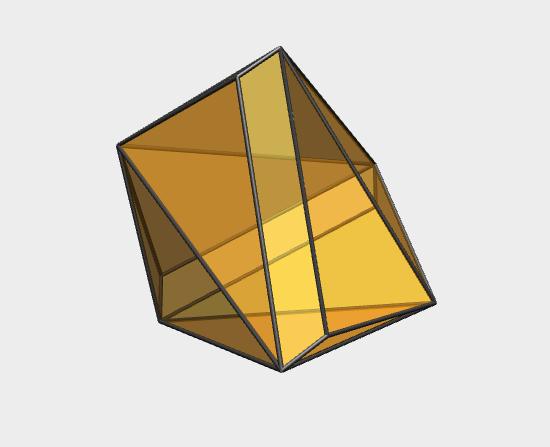}
\includegraphics[scale=0.2]{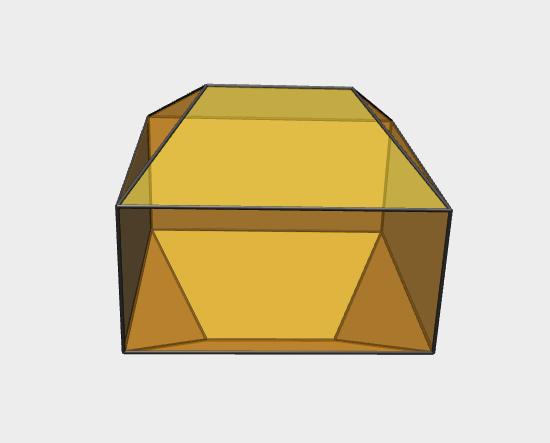}

\includegraphics[scale=0.2]{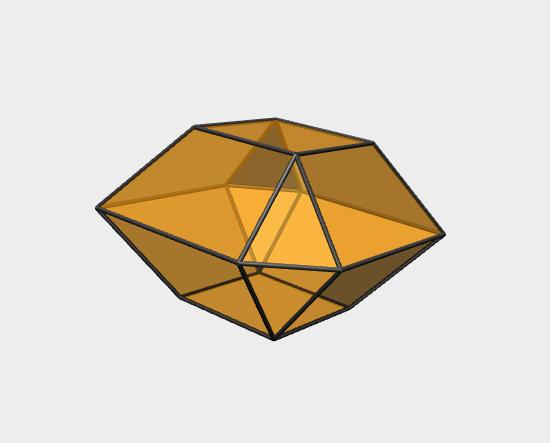}
\includegraphics[scale=0.2]{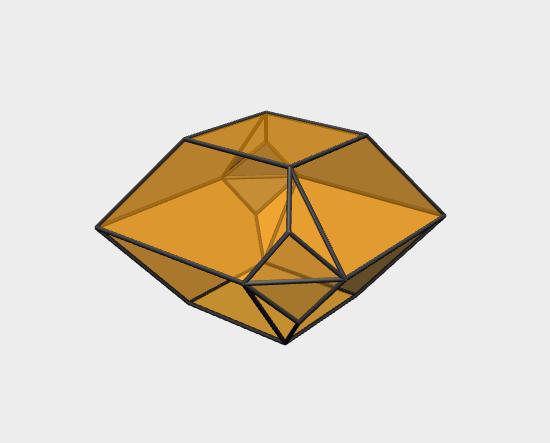}
\end{center}
\caption{The $\Dih_2$-symmetric polytopes $Dih_2(10,10)$, $Dih_2(10,14)$, $Dih_2(12,12)$, $Dih_2(14,14)$ and $Dih_2(18,18)$}\label{fig:special_dih_two}
\end{figure}

\begin{proof}
Denote by $\F'$ the right hand side of the assertion.
By Lemma \ref{lem:f_vectors_mod_n} we have
\begin{align*}
\F(\Dih_2)&\subset \{f\in\F \ : \\
			& \hspace{2em}\ (f \mod 4)\in \{(0,0),(0,2)\}^\diamond + \{(0,0),(0,2)\}^\diamond + \{(0,0),(0,2)\}^\diamond\}\\
&=\{f\in\F \ : \ f\equiv (0,0),(0,2),(2,2)\mod 4\}^\diamond.
\end{align*}
Next we show that $f(P)\neq (6,6)$.
Suppose $f(P)\equiv (2,2) \mod 4$. Then each flip orbit intersects two vertices, two edges and two facets of $P$, respectively. These vertices and facets are incident with at least $4$ edges due to the induced symmetry. Furthermore, due to the symmetry no edge intersects more than one flip-axis. Thus $f_0(P) + f_2(P) -2 = f_1(P)\geq 4\cdot 4 + 2= 18$ and therefore $f(P) \neq (6,6)$.


Altogether, this shows $F(\Dih_2)\subset F'$. To see that $F'\subset F(\Dih_2)$ consider Corollary \ref{cor:certificates_work} and the following table (note that the triangle certificate for $f = (6,6)$ is missing the top entry, which is not relevant for the existence of $f$-vectors other than $(6,6)$).
\tcs
\begin{longtable}{cc|ccc}
\hspace{1em}$(p,q)$  & \ \hspace{0.7cm} & $v_1(p,q)$ & $v_2(p,q)$ & $v_3(p,q)$\\
\hline
\hspace{1em}$(0,0)$ &
& \defRoot{$(4,4)$}\defT {$Tet$} \certB
& \defRoot{$(8,8)$}\defT {$DT$} \certB
& \defRoot{$(12,12)$}\defT {$Dih_2(12,12)$} \certB\\
\hline
\hspace{1em}$(0,2)$ &
& \defRoot{$(8,6)$}\defLR {$CubOc$} {$Cub$} \certRL
& \defRoot{$(12,10)$}\defLR {$\RP_{4,2}(TPri_4)$} {$DPri_4$} \certRL
& \defRoot{$(8,10)$}\defLR {$TPri_4$} {$EB_{6,4}$} \certLR\\
\hline
\hspace{1em}$(2,2)$ &
& \defRoot{$(6,6)$}\defLRTX {$Dih_2(10,14)$\hspace{0.5cm}} {\hspace{0.5cm}$Dih_2(10,14)^\vee$} {$\varnothing$} {$Dih_2(18,18)$}\certTri 
& \defRoot{$(10,10)$}\defT {$Dih_2(10,10)$} \certB
& \defRoot{$(14,14)$}\defT {$Dih_2(14,14)$} \certB\\
\hline
\end{longtable}
\end{proof}
Now we consider the tetrahedral rotation group $\T$.
Take a given regular tetrahedron with barycenter $0$.
The tetrahedral rotation group $\T$ contains four order three rotations around axis which go through a vertex and the respectively opposing facet.
Furthermore it contains flips around axes through the midpoints of two opposing edges.

This group has two non-regular orbits of size $4$ which are not flip-orbits and a flip-orbit of size $6$.

We do not need any further polytopes other than the polytopes in Notation~\ref{not:special_polytopes} to proof the following characterization:

\begin{theorem}
We have
\begin{align*}
\F(\T) = \{\f\in\F \ : \ \f\equiv (0,2),(0,8),(4,4),(4,10) \mod 12\}^\diamond.
\end{align*}
\end{theorem}\begin{proof}
Denote by $F'$ the right hand side of the assertion. 
By Lemma \ref{lem:f_vectors_mod_n} we have
$$\F(\T) \subset \{\f\in\F \ : \ (\f\mod 12)\in \{(0,8), (8,0), (4,4)\} + \{(0,0),(0,6),(6,0)\} $$ which is equivalent to $\F(T)\subset F'$.
To see that $F'\subset \F(T)$ consider Table \ref{tab:tetra} in the appendix.
\end{proof}


Next, we consider the octahedral rotation group $\Oc$. Take a given regular cube with barycenter $0$.
The group $\Oc$ contains three four-fold rotations around axis through opposing facets.
Furthermore it contains four three-fold rotations around axis through two opposing vertices and flips around axis through the midpoints of the edges.

Thus, the group $\Oc$ has one non-regular orbit of size $6$ consisting of the rays of the threefold rotation axes. Furthermore it has a non-regular orbit of size $8$ consisting of all rays of the fourfold rotation axes.
Lastly, it has a flip orbit of size $12$ consisting of all rays belonging to the flip axes.


\begin{theorem}
We have 
\begin{align*}
\F(\Oc) = \{\f\in\F \ : \ \f\equiv (0,2),(0,14),(6,8),(6,20),(8,18),(12,14) \mod 24\}^\diamond.
\end{align*}
\end{theorem}

Without loss of generality we may consider $\Oc$ as the group generated by

\begin{equation*}
\begin{pmatrix}
0 & -1 & 0 \\
1& 0 & 0 \\
0& 0& 1
\end{pmatrix},
\begin{pmatrix}
1 & 0 & 0 \\
0 & 0 & -1 \\
0& 1& 0
\end{pmatrix}
\end{equation*}

and state the following $\O$ symmetric polytopes in an explicit way (see also Figure~\ref{fig:special_oct}):

\begin{itemize}

\item $Oct(72,50)=\conv(\Oc\cdot \{(1,2,3) , (3,2,1), (1,0,4)\})$ 

\item $Oct(30,44)=\conv(\Oc\cdot \{(4,0,0) , (-1,2,2)\})$ 

\item $Oct(32,42)=\conv(\Oc\cdot \{(2,2,2) , (0,1,3)\})$

\item $Oct(60,38)=\conv(\Oc\cdot \{(8,8,0) , (1,7,5), (1,7,-5)\})$

\item $Oct(54,32) = (\CS_{3,8}(SnCub))^\vee$

\end{itemize}

\begin{figure}
\begin{center}
\includegraphics[scale=0.15]{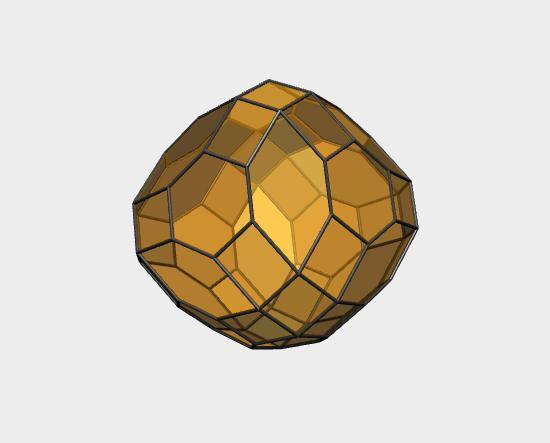}
\includegraphics[scale=0.15]{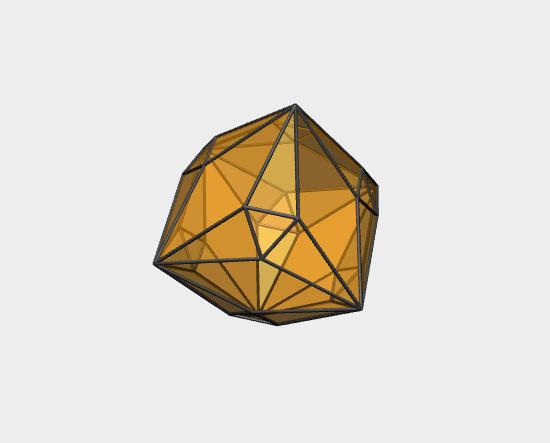}
\includegraphics[scale=0.15]{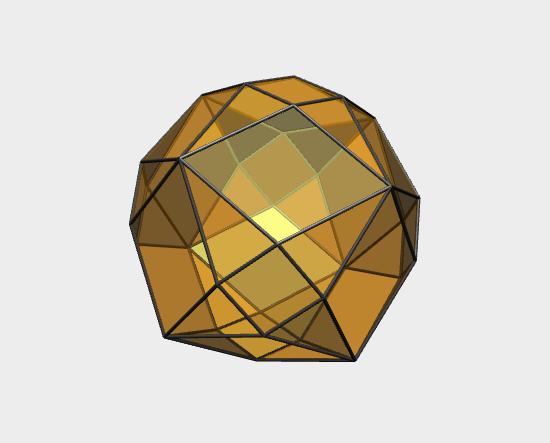}
\includegraphics[scale=0.15]{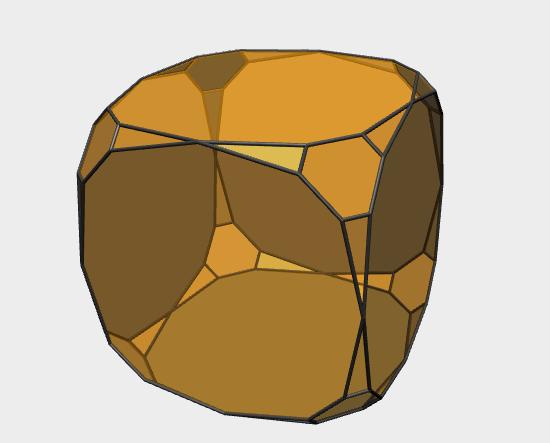}
\end{center}
\caption{The $\Oc$-symmetric polytopes $Oct(72,50)$, $Oct(30,44)$, $Oct(32,42)$, $Ot(60,38)$ and $Oct(54,32)$}\label{fig:special_oct}
\end{figure}

\begin{proof}
Denote by $\F'$ the right hand side of the assertion. 
By Lemma \ref{lem:f_vectors_mod_n} we have
$$\F(\Oc)\subset \{ f\in\F \ : \ (f \mod 12)\in \{(0,6)\}^\diamond + \{(0,8)\}^\diamond + \{(0,0),(0,12)\}^\diamond\}$$
which is equivalent to $\F(\Oc)\subset \F'$.
To see that $\F'\subset\F(\Oc)$ consider Corollary \ref{cor:certificates_work} and Table \ref{tab:octa} in the appendix.

\end{proof}

Next, we consider the icosahedral rotation group $\I$.
Take a given icosahedron with barycenter $0$.
The group $\I$ contains five-fold rotations around axes through opposing vertices. 
Furthermore, it contains three-fold rotations around axes through opposing facets and flips around axes through the midpoints of edges.

Therefore, the group $\I$ has three different non-regular orbits. One non-flip of size $20$ consisting of the rays belonging to the threefold rotation axes. Furthermore, there is a non-generate orbit of size $12$ consisting of the rays belonging to the five-fold rotations. Lastly, there is a flip-orbit consisting of the $30$ rays belonging to the flip axes.

\begin{theorem}
We have
\begin{align*}
\F(\I) = \{\f\in\F \ : \ \f\equiv (0,2),(0,32),(12,20), (12,50), (20,42),(30,32)\mod 60\}^\diamond.
\end{align*}
\end{theorem}

Let $ \Phi= \frac{1+\sqrt{5}}{2} $ be the golden ratio. Without loss of generality we consider $\I$ as the matrix group generated by

\begin{align*}
\begin{pmatrix}
-1 & 0 &  0 \\
0 & -1 & 0\\
0 & 0 & 1
\end{pmatrix},
\begin{pmatrix}
0 & 0 & 1 \\
1 & 0 & 0\\
0 & 1 & 0
\end{pmatrix},
\begin{pmatrix}
\frac 1 2  & \frac 1 2 \Phi - \frac 1 2 &  - \frac 1 2 \Phi \\
  -\frac 1 2 \Phi + \frac 1 2 & - \frac 1 2 \Phi  & -\frac 1 2\\
- \frac 1 2 \Phi  & \frac 1 2 &   -\frac 1 2 \Phi + \frac 1 2
\end{pmatrix}
\end{align*}

and state the following $\I$ symmetric polytopes in an explicit way (see also Figure~\ref{fig:special_ico}:

\begin{itemize}


\item $Ico(72, 50)=\conv(\I\cdot \{(1,0,1 ), (-\frac 1 2  \Phi -\frac 1 2, 0, - \frac 1 2  \Phi  )\})$

\item $Ico(72, 110)=\conv(\I\cdot \{(0,1,\Phi ), (-\frac 1 4, \frac 3 4 \Phi + \frac 1 4 ,  \frac  3 4  \Phi - \frac 1 2  )\})$

\item $Ico(80, 42)=\conv(\I\cdot \{(1,0,1 ), (-\frac 1 2, 0, - \frac 1 2  \Phi - \frac 1 2  )\})$

\item $Ico(150, 92)=\conv(\I\cdot \{(0,0,1 ), (\frac 1 6, \frac 2 3 \Phi - \frac 2 3,  \frac 1 6  \Phi + \frac 2 3 ), (\frac 1 {12}, \frac 7 {12} \Phi - \frac 7 {12},  \frac 1 {12} \Phi + \frac 5 6)\})$

\item $Ico(180,122) =\conv(\I\cdot \{(\frac{1}{2} \Phi + \frac 1 4, -\frac 1 4 \Phi +\frac 1 4, \frac 1 4  \Phi + \frac 1 2 ), (-\frac 1 6  \Phi +1, - \frac 1 6,\frac 1 6  \Phi + \frac 5 6 ),$\\
$(-\frac 1 3 \Phi +1 , -\frac 1 3,  \frac 1 3  \Phi + \frac 2 3 )\})$
\end{itemize}

\begin{figure}
\begin{center}
\includegraphics[scale=0.2]{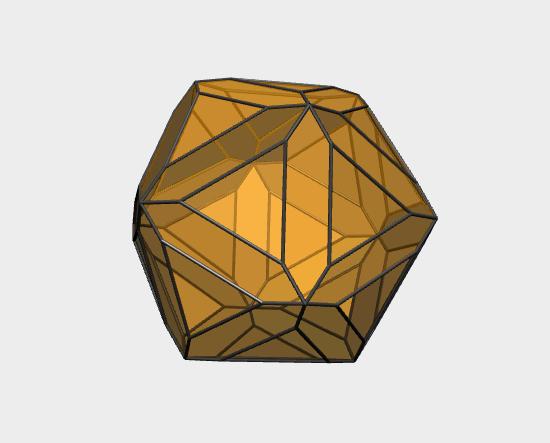}
\includegraphics[scale=0.2]{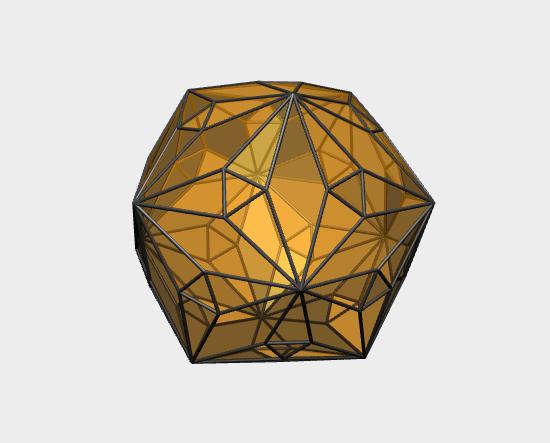}
\includegraphics[scale=0.2]{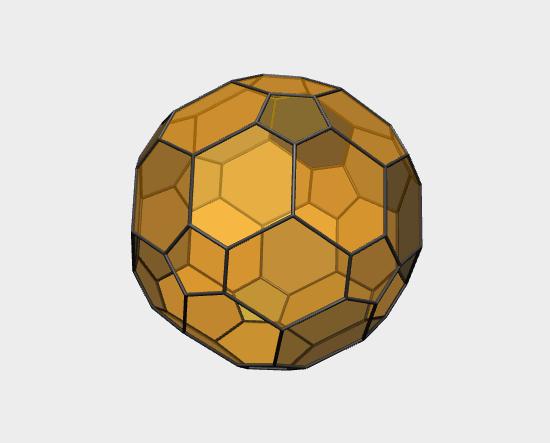}
\includegraphics[scale=0.2]{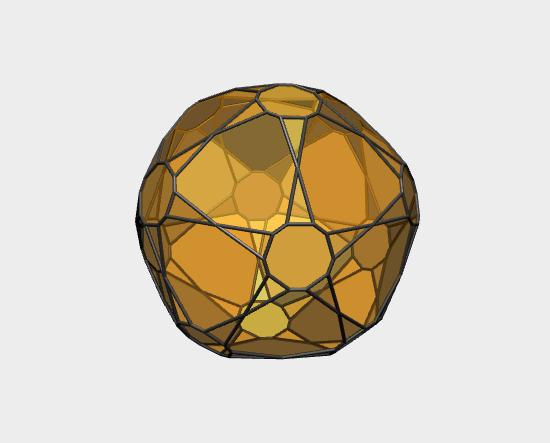}
\includegraphics[scale=0.2]{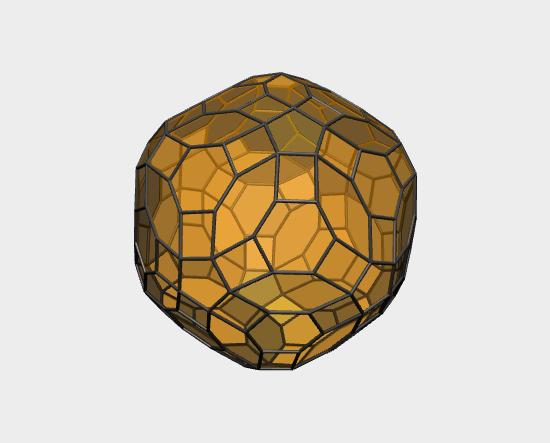}
\end{center}
\caption{The $\I$-symmetric polytopes $Ico(72,50)$, $Ico(72,110)$, $Ico(80,42)$, $Ico(150,92)$ and $Ico(180,122)$}\label{fig:special_ico}
\end{figure}

\begin{proof}

Denote by $\F'$ the right hand side of the assertion.
By Lemma \ref{lem:f_vectors_mod_n} we therefore have
$$\F(\I)\subset \{ f\in\F \ : \ (f\mod 60) \in \{(0,20)\}^\diamond + \{(0,12)\}^\diamond + \{(0,0),(0,30)\}^\diamond\},$$
which is equivalent to $\F(\I)\subset \F'$.
To see that $\F'\subset\F(\I)$ consider Corollary \ref{cor:certificates_work} and Table \ref{tab:ico} in the appendix.

\end{proof}

Now, consider the rotary reflection group $G = \G_d$. It is generated by the product $\Theta\cdot \sigma$ where $\Theta$ is a $2d$-fold rotation and $\sigma$ is a reflection orthogonal to the rotation axis of $\Theta$. The order of $\G_d$ is $n=2d$. The group has a non-regular orbit of size two on the rotation axis. All other orbits are regular.
We need no further polytopes to derive the following result:

\begin{theorem}
We have,
$$\F(\G_d) = \{\f\in\F \ : \ \f\equiv (0,2) \mod n\}^\diamond \textnormal{ for } d > 2.$$
\end{theorem}
\begin{proof}
Denote by $F'$ the right hand side of the assertion. By Lemma \ref{lem:f_vectors_mod_n} we have $F(\G_d)\subset F'$.
To see that $F' \subset F(\G_d)$ consider the following table:

\tcs
\begin{tabular}{cc|ccc}
\hspace{1em}$(p,q)$ & \ \hspace{0.7cm} & $v_1(p,q)$ & $v_2(p,q)$ & $v_3(p,q)$\\
\hline
\hspace{1em}$(0,2)$ &
& \defRoot{$(2n,n+2)$}\defLR {$\RP_{d,2}^2$} {$Pri_{2d}$} \certRL
& \defRoot{$(n,n+2)$}\defLR {$TPri_d$} {$DP_{2d}$} \certLR
& \defRoot{$(2n,2n+2)$}\defT {$\RP_{d,2}(TPri_d)$} \certB \hspace{1em}\\
\end{tabular}
\end{proof}

Next, we consider the group $G_2$ where the rotation axis provides a flip-orbit of size two. As a special $G_2$-symmetric polytope we consider $G_d(16,18)$ the square orthobi cupola also known as Johnson solid $28$.

\begin{theorem}
We have
$$F(\G_2) = \{f\in \F \ : \ f\equiv (0,0),(0,2) \mod 4\}.$$
\end{theorem}
\begin{proof}
Denote by $F'$ the right hand side of the assertion. By Lemma \ref{lem:f_vectors_mod_n} we have $F(\G_2)\subset F'$. To see that $F'\subset F(\G_2)$ consider the following table:

\tcs
\begin{tabular}{cc|ccc}
\hspace{1em} $(p,q)$ & \ \hspace{0.7cm} & $v_1(p,q)$ & $v_2(p,q)$ & $v_3(p,q)$\\
\hline
\hspace{1em} $(0,0)$ &
& \defRoot{$(4,4)$} \defT {$Tet$} \certB
& \defRoot{$(8,8)$} \defT {$DT$} \certB
& \defRoot{$(12,12)$} \defT {$Dih_2(12,12)$} \certB \\
\hline
\hspace{1em} $(0,2)$ &
& \defRoot{$(8,6)$} \defLR {$CubOc$} {$Cub$} \certRL
& \defRoot{$(12,10)$} \defLR {$G_d(16,18)$} {$DPri_4$} \certRL
& \defRoot{$(8,10)$} \defLR {$TPri_4$} {$\RP_{4,2}(Cub)$} \certLR\hspace{1em}\\
\end{tabular}

\end{proof}

The group $\G_1$ is the point reflection at the origin. As a matrix group it is generated by the negative identity matrix. 

\begin{theorem}
We have
$$F(\G_1) =  \{\f\in\F \ : \ \f\equiv (0,0) \mod 2\}^\diamond\setminus \{(4,4),(6,6)\}.$$
\end{theorem}

For the proof, we give the following $\G_1$-symmetric polytopes in an explicit way (see also Figure~\ref{fig:special_point_reflection}):

\begin{itemize}

\item $PRefl(8,10)=\conv(\G_1\cdot \{(5,0,0) , (0,5,0), (0,0,5),(3, -\frac 1{2},\frac 52)\})$ 

\item $PRefl(10,10)=\conv(\G_1\cdot \{(4,0,0), (0,4,0) ,(-1, 1,4), (1, -1,4) ,(2, -3,2)\})$ 

\end{itemize}

\begin{figure}
\begin{center}
\includegraphics[scale=0.2]{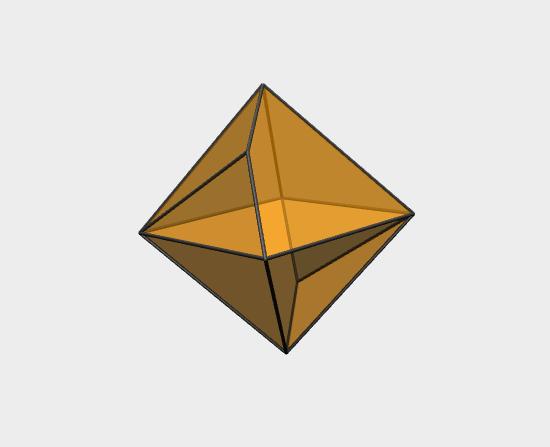}
\includegraphics[scale=0.2]{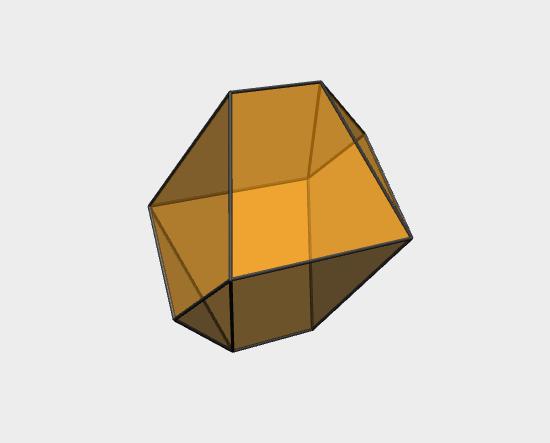}
\end{center}
\caption{The $G_1$-symmetric polytopes $PRefl(8,10)$ and $PRefl(10,10)$}\label{fig:special_point_reflection}
\end{figure}

\begin{proof}
Denote by $F'$ the right hand side of the assertion. 
First we show that $ (4,4), (6,6)\notin F(\G_1)$. Note that any facet and its reflection at the origin do not intersect. 
From that we can first conclude that a $\G_1$-symmetric polytope has at least 6 vertices. 
Secondly,  a $G_1$-symmetric polytope with 6 vertices has to be simplicial. By Theorem \ref{thm:Euler_Steinitz}, that means $2f_0-f_2=4$ which is not satisfied for $f=(6,6)$.
Together with  Lemma \ref{lem:f_vectors_mod_n} we thus have $F(\G_1)\subset F'$.

To see that $F'\subset F(\G_1)$ consider Corollary \ref{cor:certificates_work} and the following table:
\tcs
\begin{tabular}{cc|ccc}
\hspace{1em} $(p,q)$ & \ \hspace{0.7cm} & $v_1(p,q)$ & $v_2(p,q)$ & $v_3(p,q)$\\
\hline
\hspace{1em} $(0,0)$ &
& \defRoot{$(4,4)$}\defLRTX {$Cub$} {$Oc$} {$\varnothing$}{$PRefl(10,10)$}\certTri
& \defRoot{$(6,6)$}\defLRTX {$PRefl(8,10)$} {\hspace{12pt}$PRefl(8,10)^\vee$} {$\varnothing$}{$\CC_{3,2}(PRefl(8,10))$}\certTri
& \defRoot{$(8,8)$} \defT {$DT$} \certB \hspace{1em} \\
\end{tabular}

\end{proof}

This finishes the proof of Theorem \ref{thm:main}.

\section{Open questions}\label{sec:open}

Up to this point only reflection free symmetries have been discussed. 
To prove the main theorem, Theorem \ref{thm:main}, we mostly followed the characterization \ref{thm:finite_orthogonal_groups} from Grove and Benson \cite[Theorem 2.5.2]{benson1985finite_reflection_groups}. 
In order to characterize the $f$-vectors of the remaining symmetry groups, it is important to know about the contained reflections
and their arrangement. Therefore, we propose the following characterization of symmetry groups containing reflections instead:

Denote a rotation around the $z$-axis by an angle of $2\pi/d$ by $\Theta_d$ and a reflection in a plane $H$ by $\sigma_H$. Let further $z^\perp$ and $x^\perp$  be the planes through the origin orthogonal to the $z$-axis and the $x$-axis, respectively. 
Then we have the following:

\begin{theorem}\label{thm:alternative_characterization_reflection_groups}
Let $G$ be a finite orthogonal subgroup of $GL_3(\R)$.
If $G$ contains a reflection  then it is isomorphic to one of the following:
\begin{enumerate}
\item The cyclic rotation group with an additional reflection orthogonal to the rotation axis $C_d^{\perp} \simeq\left <\Theta_d, \sigma_{z^\perp} \right>$, $d\geq 1$,
\item the cyclic rotation group with an additional reflection that contains the rotation axis $C_d^{\subset} \simeq \left <\Theta_d, \sigma_{x^\perp} \right>$, $d\geq 2$,
\item the cyclic rotation group containing both additional reflections \\  $C_d^{\perp,\subset}~\simeq~\left <\Theta_d, \sigma_{z^\perp}, \sigma_{x^\perp} \right>$, $d\geq 2$,
\item the rotary reflection group of order $2d$ with an additional reflection containing the rotation axis $G_d^{\subset} = \left <-\Theta, \sigma_{x^\perp}\right>$ where $\Theta = \Theta_{2d}$ if $d$ is even and $\Theta = \Theta_d$ if $d$ is odd, $d\geq 2$
\item the full tetrahedral, octahedral and icosahedral symmetry group $\hat\T$, $\hat\O$, $\hat\I$ respectively,
\item the group $\T^\ast = \T\cup \{-X \ : \ X\in\T\}$.
\end{enumerate}
\end{theorem}

\begin{proof}
In \cite[Theorem 2.5.2]{benson1985finite_reflection_groups} it is shown that any finite orthogonal group consists of rotations and negatives of rotations.
Observe that $\Theta$ is a flip around an axis $v$ if and only if $-\Theta$ is a reflection on $v^\perp$.
Furthermore, $\left <-\Theta_d \right>$ contains a reflection if and only if $d\equiv 2 \mod 4$.
If $d\not\equiv 2 \mod 4$ this group has order $d$ for even $d$ and order $n = 2d$ for odd $d$.
Using this, we can compare the above list with the characterization to observe that these characterizations are equivalent.
In particular (in the notation of Benson)
\begin{enumerate}
\item corresponds to $(C_3^{d})^\ast$ for even $d$ and to $C_3^{2d} ] C_3^d$ for odd $d$,
\item corresponds to $\mathcal H_3^d ] C_3^d$,
\item corresponds to $(\mathcal H_3^{d})^\ast$ for even $d$ and $\mathcal H_3^{2d} ] \mathcal H_3^d$ for odd $d$,
\item corresponds to $\mathcal H_3^{2d} ] \mathcal H_3^{d}$ for even $d$ and to $(\mathcal H_3^{d})^\ast$ for odd $d$,
\item corresponds to $\mathcal W ] \mathcal T$, $\mathcal W ^\ast$ and $\mathcal I^\ast$,
\item corresponds to $\mathcal T^\ast$,
\end{enumerate}
\end{proof}

By applying Lemma \ref{lem:f_vectors_mod_n} we already have the following :
\begin{enumerate}
\item $F(C_d^\perp) \subset \{f\in F \ : \ f\equiv (0,2) \mod d\}^\diamond$ for $d>2$, \\
 $F(C_2^\perp)\subset \{f\in F \ : \ f\equiv (0,0) \mod 2\}^\diamond$,\\
  $F(C_1^\perp) \subset F$ ,
\item $F(C_d^\subset) \subset \{f\in F \ : \ f\equiv (0,2),(1,1) \mod d\}^\diamond$ for $d>2$, \\
 $ F(C_2^\subset) \subset F$,
\item $F(C_d^{\perp,\subset}) \subset \{f\in F \ : \ f\equiv (0,2) \mod d\}^\diamond$, $d>2$,\\
 $F(C_2^{\perp,\subset}) \subset \{f\in F \ : \ f\equiv (0,0) \mod 2\}^\diamond$,
\item $F(\G_d^\subset) \subset \{f\in F \ : \ f\equiv (0,2) \mod 2d\}^\diamond$, $d>2$,\\
 $F(\G_2^\subset) \subset \{f\in F \ : \ f\equiv (0,0), (0,2) \mod 4\}^\diamond$
\item $F(\hat\T) \subset F(\T)$, $F(\hat\O) \subset F(\O)$, $F(\hat\I) \subset F(\I)$,
\item $F(\T^\ast) \subset \{f\in F \ : \ f\equiv (0,2),(0,8),(6,8) \mod 12\}^\diamond$.
\end{enumerate}

Next, we consider another related problem. For any three dimensional polytope $P$ we define its linear symmetry group by $\Symm(P) = \{ A\in \R^{3\times 3} \ : \ A\cdot P = P\}$. Furthermore, we denote
$$\overline{F(G)} = \{ f \in F \ : \ \text{there is a polytope $P$ with} f(P) = f, \Symm(P) = G\}.$$
Then $\overline{F(G)}\subset F(G)$. We conjecture that every 'large enough' $f$-vector in $F(G)$ is also contained in $\overline{F(G)}$ while there are finitely many $f$-vectors in $F(G)\backslash\overline{F(G)}$. This conjecture is indicated by \cite{Isaacs1977linear_groups_as_stabilizers_of_sets}.


Another interesting problem is the problem in higher dimensions, even the traditional $f$-vector problem is very hard in four dimensions. Nevertheless, the constructions in Corollary \ref{cor:operations} and Lemma \ref{lem:f_vectors_mod_n} can be generalized for any dimension yielding an inner and outer approximation of $f$-vectors of symmetric polytopes. It is expectable that these approximations differ a lot.

To conclude this paper we summarize open problems which deserve further investigation:
\begin{enumerate}
\item what is $F(G)$ if $G$ is one of the groups described in Theorem \ref{thm:alternative_characterization_reflection_groups}?
\item What is $\overline{F(G)}$?
\item Is it possible to find good inner and outer approximations of $F(G)$ in dimension~$4$?
\item What do we know in arbitrary dimensions?
\item What kind of flag-vectors are possible for symmetric 3-polytopes?
\end{enumerate}

\section*{Acknowledgement}
We like to thank Frieder Ladisch and Michael Joswig for many helpful advices. 

\bibliography{f_vectors}
\bibliographystyle{alpha}

\newpage
\section*{Appendix}

\tcs
\begin{longtable}{cc|ccc}
\caption{Certificates for the tetrahedral rotation group $\T$.}\\
 \label{tab:tetra}
$(p,q)$  & \ \hspace{0.7cm} & $v_1(p,q)$ & $v_2(p,q)$ & $v_3(p,q)$\\
\hline
$(0,2)$ &
& \defRoot{$(24,14)$}\defLR {$\TP_{3,4}(*)$} {$TrCub$} \certRL
& \defRoot{$(12,14)$}\defLRTX {$\TP_{3,4}(*)$} {$\RP_{3,4}(TrCub)$} {$CubOc$} {$\TP_{3,4}\circ\RP_{3,4}(TrCub)$}\certTri
& \defRoot{$(24,26)$}\defLRTX {$\TP_{3,4}(\ast)$\hspace{0.5cm}} {\hspace{0.5cm}$\RP^2_{3,4}(TrCub)$} {$RCubOc$} {$\RP_{3,8}(\ast)$}\certTri\\
\hline
$(0,8)$ &
& \defRoot{$(12,8)$}\defLR {$\TP_{3,4}(*)$} {$TrTet$} \certRL
& \defRoot{$(24,20)$}\defLR {$\TP_{3,4}(\ast)$} {$\RP_{3,4}(TrTet)$} \certRL	
& \defRoot{$(12,20)$}\defLRTX {$\THP_{6,4}(TrTet)$} {\hspace{1cm}$\RP^2_{3,4}(TrTet)$} {$Ico$} {$\RP^2_{3,4}\circ\TP_{3,4}(TrTet)$}\certTri\\
\hline
$(4,4)$ &
& \defRoot{$(4,4)$}\defLRTX {$\TP_{3,4}(*)$} {$\TP^\vee_{3,4}(*)$} {$Tet$} {$\RP_{3,4}^3(*)$}\certTri
& \defRoot{$(16,16)$}\defT {$\RP_{3,4}^\vee(Tet) \ \RP_{3,4}(Tet) $} \certB
& \defRoot{$(28,28)$}\defT {$(\RP_{3,4}^\vee)^2 (Tet) \ , \ \RP_{3,4}^2(Tet) $} \certB\\
\hline
$(4,10)$ &
& \defRoot{$(16,10)$}\defLR {$\TP_{3,4}(*)$} {$\CC_{3,4}(Cub)$} \certRL
& \defRoot{$(28,22)$}\defLR {$\TP_{3,4}(*)$} {$\RP_{3,4}\circ \CC_{3,4}(Cub)$} \certRL
& \defRoot{$(16,22)$}\defLR {$\CS_{3,4}(CubOc)$} {$\RP^2_{3,4}\circ \CC_{3,4}(Cub)$} \certLR\\
\hline
$(6,8)$ &
& \defRoot{$(6,8)$} \defLRTX  {$\TP_{3,4}(\ast)$} {$\CC_{3,4}^2(RDo)$} {$Oc$} {$\RP^3_{3,4}(\ast)$} \certTri
& \defRoot{$(18,20)$} \defLR {$\TP_{3,4}(\ast)$} {$\RP_{3,4}(Oc)$} \certRL
& \defRoot{$(30,32)$} \defLR {$\TP_{3,4}(\ast)$} {$\RP_{3,4}^2(Oc)$} \certRL\\
\hline

\end{longtable}

\tcs
\begin{longtable}{cc|ccc} 
\caption{Certificates for the octahedral rotation group $\O$.}\\
\label{tab:octa}
$(p,q)$ & \ \hspace{0.7cm} & $v_1(p,q)$ & $v_2(p,q)$ & $v_3(p,q)$\\
\hline
$(0,2)$ &
& \defRoot{$(48,26)$}\defLR {$\HP_{6,8}(*)$} {$TrCubOc$} \certRL
& \defRoot{$(24,26)$}\defLRTX {$\TP_{3,8}(*)$} {$Oct(72,50)$} {$RCubOc$} {$\HP_{6,8}(Oct(72,50))$}\certTri
& \defRoot{$(48,50)$}\defLR {$\TP{3,8}(*)$} {$\RP_{3,8}(RCubOc)$} \certRL\\
\hline
$(0,14)$ &
& \defRoot{$(24,14)$}\defLR {$\TP_{3,8}(*)$} {$TrCub$} \certRL
& \defRoot{$(48,38)$}\defLR {$\TP_{3,8}(*)$} {$\RP_{3,8}(TrCub)$} \certRL
& \defRoot{$(24,38)$}\defLR {$SnCub$} {$\RP^2_{3,8}(TrCub)$} \certLR\\
\hline
$(6,8)$ &
& \defRoot{$(6,8)$}\defLRTX {$\TP_{3,8}(*)$} {$Oct(54,32)$} {$Oc$} {$\RP_{3,8}^3(\ast)$}\certTri
& \defRoot{$(30,32)$}\defT {$\RP_{3,8}^\vee(Oc) \ \RP_{3,8}(Oc) $} \certB
& \defRoot{$(54,56)$}\defT {$(\RP_{3,8}^\vee)^2(Oc) \ \RP_{3,8}^2(Oc) $} \certB\\
\hline
$(6,20)$ &
& \defRoot{$(30,20)$}\defLR {$\TP_{3,8}(*)$} {$\CC_{3,8}(RDo)$} \certRL
& \defRoot{$(54,44)$}\defLR {$\TP_{3,8}(*)$} {$\RP_{3,8}\circ\CC_{3,8}(RDo)$} \certRL
& \defRoot{$(30,44)$}\defLR {$Oct(30,44)$} {$\RP^2_{3,8}\circ\CC_{3,8}(RDo)$} \certLR\\
\hline
$(8,18)$ &
& \defRoot{$(32,18)$}\defLR {$\TP_{4,6}(*)$} {$\CC_{4,6}(RDo)$} \certRL
& \defRoot{$(56,42)$}\defLR {$\TP_{4,6}(*)$} {$\RP_{4,6}\circ\CC_{4,6}(RDo)$} \certRL
& \defRoot{$(32,42)$}\defLR {$Oct(32,42)$} {$\RP^2_{4,6}\circ\CC_{4,6}(RDo)$} \certLR\\
\hline
$(12,14)$ &
& \defRoot{$(12,14)$}\defLRTX {$\TP_{3,8}(*)$} {$Oct(60,38)$} {$CubOc$} {$\RP_{3,8}^3(\ast)$}\certTri
& \defRoot{$(36,38)$}\defLR {$\TP_{3,8}(*)$} {$\RP_{3,8}(CubOc)$} \certRL
& \defRoot{$(60,62)$}\defLR {$\TP_{3,8}(*)$} {$\RP^2_{3,8}(CubOc)$} \certRL\\
\hline

\end{longtable}

\tcs
\begin{longtable}{cc|ccc}
\caption{Certificates for the icosahedral rotation group $\I$.}\\
\label{tab:ico}
$(p,q)$  & \ \hspace{0.7cm} & $v_1(p,q)$ & $v_2(p,q)$ & $v_3(p,q)$\\
\hline
$(0,2)$ &
& \defRoot{$(120,62)$}\defLR {$\HP_{6,20}(*)$} {TrID} \certRL
& \defRoot{$(60,62)$}\defLRTX {$\TP_{3,20}(*)$} {Ico(180,122)} {$RID$} {$\RP^3_{5,12}(*)$}\certTri
& \defRoot{$(120,122)$}\defLR {$\TP_{3,20}(*)$} {$\RP_{5,12}(RID)$} \certRL\\
\hline
$(0,32)$ &
& \defRoot{$(60,32)$}\defLR {$\HP_{6,20}(*)$} {$TrI$} \certRL
& \defRoot{$(120,92)$} \defLR {$\HP_{6,20}(*)$} {$\RP_{5, 12}(TrI)$} \certRL 
& \defRoot{$(60,92)$}\defLR {$SnDo$} {$\RP^2_{5,12}(TrI)$} \certLR\\
\hline
$(12.20)$ &
& \defRoot{$(12,20)$}\defLRTX {$\TP_{3,20}(*)$} {$\TP_{5,12}^\vee(*)$} {$Ico$}{$\RP^3_{3,20}(*)$}\certTri
& \defRoot{$(72,80)$}\defLR {$\TP_{3,20}(*)$} {$\RP_{3,20}(Ico)$} \certRL
& \defRoot{$(132,140)$}\defLR {$\TP_{3,20}(*)$} {$\RP^2_{3,20}(Ico)$} \certRL\\
&\\
\hline
$(12,50)$ &
& \defRoot{$(72,50)$}\defLR {$\TP_{3,20}(*)$} {Ico(72,50)} \certRL
& \defRoot{$(132,110)$}\defLR {$\TP_{3,20}(*)$} {$\RP_{3,20}(Ico(72,50))$} \certRL
& \defRoot{$(72,110)$}\defLR {Ico(72,110)} {$\RP_{3,20}^2(Ico(72,50))$} \certLR\\
\hline
$(20,42)$ &
& \defRoot{$(80,42)$}\defLR {$\TP_{5,12}(\ast)$} {Ico(80,42)}\certRL
& \defRoot{$(140,102)$}\defLR {$\TP_{5,12}(*)$} {$\RP_{5,12}(Ico(80,42))$}\certRL
& \defRoot{$(80,102)$}\defLR {$\CS_{3,20}(RID)$} {$\TP^\vee_{3,20}(*)$}\certLR\\
\hline
$(30,32)$ &
& \defRoot{$(30,32)$}\defLRTX {$\TP_{3,20}(*)$} {Ico(150,92)} {$ID$} {$\CS_{3,60}(Ico(150,92))$}\certTri
& \defRoot{$(90,92)$}\defLR {$\TP_{5,12}(*)$} {$\RP_{5,12}(ID)$}\certRL
& \defRoot{$(150,152)$}\defLR {$\TP_{5,12}(*)$} {$\RP^2_{5,12}(ID)$}\certRL
\end{longtable}

%

\end{document}